\documentclass[12pt]{article}

\usepackage{amsmath,amsfonts,amstext,amssymb,amsbsy,amsopn,eucal}
\usepackage{amsthm}
\usepackage{dsfont}
\usepackage{graphicx}   

\begin{document}

\newtheorem{thm}{Theorem}[section]
\newtheorem*{mtheorem}{Main Theorem}

\newtheorem{prop}{Proposition}[section]
\newtheorem*{propA}{Proposition A}
\newtheorem*{propB}{Proposition B}
\newtheorem*{propC}{Proposition C}

\newtheorem{lem}{Lemma}[section]

\newtheorem{cor}{Corollary}[section]

\theoremstyle{definition}
\newtheorem{definition}{Definition}[section]

\theoremstyle{remark}
\newtheorem{remark}{Remark}[section]

\setcounter{equation}{0}

\title{Noncompact shrinking four solitons with nonnegative curvature}

\author{By \textit{Aaron Naber}}

\date{\today}
\maketitle
\begin{abstract}
We prove the following:  Let $(M,g,X)$ be a noncompact four dimensional shrinking soliton with bounded nonnegative curvature operator, then $(M,g)$ is isometric to $\mathds{R}^{4}$ or a finite quotient of $S^{2}\times\mathds{R}^{2}$ or $S^{3}\times\mathds{R}$.  In the process we also show that a complete shrinking soliton $(M,g,X)$ with bounded curvature is gradient and $\kappa$-noncollapsed and the dilation of a Type I singularity is a shrinking soliton.  Further in dimension three we show shrinking solitons with bounded curvature can be classified under only the assumption of $Rc\geq 0$.
\end{abstract}

\section{Introduction}
\numberwithin{equation}{section}

The study of solitons has become increasingly important in both the study of the Ricci Flow and in metric measure theory.  Solitons play a direct role as singularity dilations in the Ricci Flow proof of uniformization, see \cite{ChKn}, and more recently Perelman proved shrinking solitons play a role in the analysis of finite time singularities of all three dimensional Ricci Flows.  Under assumptions of nonnegative curvature and $\kappa$-noncollapseness Perelman classified three dimensional gradient shrinking solitons \cite{P} for the purpose of studying such singularities.  Given the importance of shrinking solitons for singularity dilations there has been much effort, see \cite{Ni} \cite{NiWa} \cite{NiWa2} for instance, to extend this classification to higher dimensions.  The main results of this paper are to extend Perelman's result to the four dimensional case and prove some structural theorems about shrinking solitons in any dimension.  We will prove that there is an \textit{a priori} lower injectivity radius bound on shrinking solitons which depends only on the curvature, soliton constant, and its $f$-volume.  It was proved by Perelman in \cite{P} that a compact shrinking soliton is always gradient, and we will extend this to noncompact shrinking solitons.  Also as an application of the estimates we will see that a Type I singularity always dilates to a shrinking soliton.  To make these statements more precise we begin with some definitions:

\begin{definition}
Let $(M,g,X)$ be a smooth complete Riemannian manifold with $X$ a complete vector field. We call $M$ a Ricci soliton if $Rc+\frac{1}{2}\mathcal{L}_{X}g=\lambda g$, where $\lambda\in\mathds{R}$. We say the soliton is shrinking,
steady, or expanding when $\lambda >0, =0, <0$, respectively.  If $(M,g,f)$ is a smooth Riemannian manifold with $f$ a smooth function such that $(M,g,\nabla f)$ is a soliton, we call $(M,g,f)$ a gradient soliton and $f$ the soliton function.
\end{definition}

We note that the soliton function for a soliton $(M,g,f)$ is well defined only up to a linear function.  By tracing the soliton equation $Rc+\nabla^{2}f=\lambda g$ with $\nabla f$ and with $g$ we get the equations $R+|\nabla f|^{2}+2\lambda f=const$ and $R+\triangle f =n\lambda$.  Combining shows that any soliton function satisfies an equation of the form $\triangle f -|\nabla f|^{2} + 2\lambda f=constant$, where the constant depends on the normalizing choice of $f$.  Motivated by this and what is to come we make the definition:

\begin{definition}
Let $(M,g,f)$ be a smooth soliton.  By rescaling $g$ and changing $f$ by a constant we can assume $\lambda \in \{-\frac{1}{2}, 0, \frac{1}{2}\}$ and $\triangle f -|\nabla f|^{2} + 2\lambda f=n\lambda$.  We call such a soliton normalized, and $f$ a normalized soliton function.  We define the $f$-volume of a soliton by $Vol_{f}(M)=\int_{M}e^{-f}dv_{g}$
\end{definition}

Although being a soliton is a purely static condition, the soliton
structure is closely related to the Ricci flow equation.  It turns
out some of the properties of a soliton are most natural to exploit
in this context, and though the following is standard we present it
for convenience because we will use it often.

\begin{lem}
Let $(M,g,f)$ be a normalized shrinking soliton on a complete manifold
with bounded curvature.  Then there exists a Ricci flow $(M,g(t))$,
$t\in (-\infty,0)$ with $g(-1)=g$
\end{lem}

\begin{proof}
Since the curvature is bounded so is $|\nabla^{2}f|$, and in
particular $X=\nabla f$ is a complete vector field.  Let
$\phi_{t}:M\rightarrow M$ be the diffeomorphisms such that $\phi_{-1}=id$
and $\frac{d}{dt}\phi_{t}(x)= Y_{t}(x) = \frac{1}{-t}X(x)$.  Let
$g(t)=(-t)\phi_{t}^{*}g$.  Then $\frac{d}{dt}g(t) = -\phi_{t}^{*}g
+(-t)\phi_{t}^{*}(\mathcal{L}_{Y}g) = -2\phi_{t}^{*}(\frac{1}{2}g-\frac{1}{2}\mathcal{L}_{X}g)
= -2\phi_{t}^{*}Rc[g] = -2Rc[g(t)]$
\end{proof}

For a normalized shrinking soliton we call the above the associated
Ricci flow.  We will often go back and forth without worry.  We could of course have just as easily used a nongradient soliton in the above.  For notational sake we will often denote $\tau=-t$.

One more definition that will be used frequently is the notion of being noncollapsed:

\begin{definition}
Let $(M,g(t))$ be a Ricci flow, $t\in [0,T]$.  Let $\kappa>0$.  We say
the Ricci flow is $\kappa$-noncollapsed if for $\forall$ $(x,s)\in
M\times(0,T]$ and $r>0$ such that the parabolic ball of radius $r$, $P(x,s,r)\equiv
B_{g(s)}(x,r)\times(s-r^{2},s]$, has compact closure contained in $M\times(0,T]$ and $|Rm|_{g(s)}\leq r^{-2}$
in $P(x,s,r)$, then $Vol(B_{g(s)}(x,r))\geq \kappa r^{n}$, where
$n=dim(M)$.
\end{definition}

The main theorem of this paper is the following:

\begin{mtheorem}\label{mthm1}
Let $(M,g,X)$ be a noncompact four dimensional shrinking soliton
with bounded nonnegative curvature operator, then $(M,g)$ is isometric to  $\mathds{R}^{4}$, or a finite quotient of $S^{2}\times\mathds{R}^{2}$ or $S^{3}\times\mathds{R}$.
\end{mtheorem}

Along the way we will also prove the following theorems

\begin{thm}\label{mthm2}
Let $(M,g,X)$ be a three dimensional shrinking soliton with bounded curvature and $Rc\geq 0$.  Then $(M,g)$ is isometric to $\mathds{R}^{3}$ or to a finite quotient of $S^{3}$ or $\mathds{R}\times S^{2}$.
\end{thm}

\begin{remark}
The above was proved by Perelman in \cite{P} under the assumptions that the soliton is gradient, $\kappa$-noncollapsed and $sec\geq 0$.  The dropping of  the $\kappa$-noncollapsed assumption follows from the next theorem, while the reducing of nonnegative sectional to nonnegative Ricci follows by a different splitting lemma at infinity and a new estimate on the mean curvature of soliton hypersurfaces.  Under the additional assumption of being gradient, though not $\kappa$-noncollapsed, the above was proved in \cite{NiWa2} by using techniques more in line with maximum principles.
\end{remark}

\begin{thm}\label{mthm3}
Let $(M,g,X)$ be a shrinking soliton with bounded curvature.  Then there exists a smooth function $f:M\rightarrow \mathds{R}$ such that $(M,g,f)$ is a gradient shrinking soliton.
\end{thm}

\begin{remark}
In the compact case this was proved in \cite{P}.  It is worth noting this is of course the best that can be said, in that a shrinking soliton $(M,g,X)$ may itself not be gradient, the above just states that there is a gradient structure on $(M,g)$.  For instance, take your favorite gradient shrinking soliton $(M,g,f)$, let $X$ be a nonparallel nontrivial killing field on $M$.  Then if $Y=\nabla f+X$ we see $(M,g,Y)$ is a nongradient shrinking soliton.  Of course the above states that any nongradient shrinking soliton has this form.
\end{remark}

\begin{thm}\label{mthm4}
Let $(M,g,f)$ be a normalized shrinking soliton with bounded curvature.  Then there exists $\kappa=\kappa(n,Vol_{f}(M))$ such that the associated Ricci flow is $\kappa$-noncollapsed.
\end{thm}

\begin{remark}
The result should be compared to a similar result for Einstein manifolds with positive Einstein constant.  There an argument using Bishop-Gromov tells us that the manifold is noncollapsed with a constant depending only on the volume and the Einstein constant.  Similarly for a shrinking soliton we will see the soliton is noncollapsed for a constant depending only on the $f$-volume and soliton constant.
\end{remark}

To state our final theorem we need the following definition

\begin{definition}
Let $(M,g(t))$ be a complete Ricci flow on a maximal time interval $[0,T)$.  We say $(M,g(t))$ encounters a Type I singularity if $\exists C>0$ such that $|Rm[g(t)]|\leq\frac{C}{|T-t|}$.
\end{definition}

\begin{thm}\label{mthm5}
Let $(M,g(t))$, $t\in [0,T)$ be a complete Ricci flow which encounters a Type I singularity at $T$.  Let $t_{i}\rightarrow T$ and $x\in M$.  With $\tau_{i}=T-t_{i}$ and $g_{i}(t)=\tau_{i}^{-1}g(\tau_{i}(T-t))$ then $(M,g_{i}(t),(x,-1))\rightarrow (N,h(t),(p,-1))$, a normalized $\kappa$-noncollapsed shrinking soliton with bounded curvature.
\end{thm}

\begin{remark}
We could let $x\in M$ from above vary with $i$ so long as $x_{i}$ doesn't tend to infinity in an appropriate sense.  A similar result was obtained in \cite{S} under the assumption that the blow up limit is compact.
\end{remark}

The proof is organized as follows.  In Section 2 we introduce a class of ancient Ricci flows with certain useful curvature properties.  This class of Ricci flows includes, among others, the associated Ricci flows of shrinking solitons and smooth limits of sequences of shrinking solitons.  We will begin by studying reduced length functions, as introduced by Perelman \cite{P}, as well as a slight generalization which behaves as reduced a length function from a singular point on these spaces.  The tools proved will be used to prove Theorems \ref{mthm3},\ref{mthm4}, \ref{mthm5} in Sections 2 and 3.  The main technical tool is Theorem \ref{asym_sol}, which will prove the existence of asymptotic solitons at both the singular time and at negative infinity for this class of Ricci flows.  In Section 4 we will use these tools to study the behavior of general noncompact shrinking solitons at infinity. Additionally Section 4 will prove a splitting lemma for arbitrary shrinking solitons with bounded curvature.  The result is similar to one proved in \cite{P}, but by not relying on the Toponogov theorem does not require a nonnegative sectional curvature assumption.

In Section 5 we take a detour and study the level sets of the soliton functions themselves.  After proving some basic properties about them we will prove an estimate on the mean curvature of such level sets which requires only a nonnegative Ricci assumption.  This is similar to estimates in \cite{P},\cite{MT} but these estimates required nonnegative sectional curvature and were only applicable in dimension three.  Using this and the tools of Sections 3 and 4 we will give a proof of the classification of shrinking solitons in dimension three which requires only a nonnegative Ricci assumption.

Section 6 will use the previous sections to give a full classification of the behavior of shrinking solitons at infinity.  Then Section 7 is dedicated to proving a technical lemma which will be useful in proving the main theorem.  We will show a shrinking soliton that satisfies $0\leq Rc$ and $\nabla^{2}f>0$ must be isometric $\mathds{R}^{n}$.  Sections 8 and 9 are then dedicated to finishing the proof of the main theorem.

\section{Controlled Ricci Flows}

We begin by pointing out that the main use of normalizing the soliton function is in the control of the $f$-volume of the manifold:

\begin{lem}\label{f_vol}
Let $(M,g,f)$, $(M',g',f')$ two normalized solitons with finite $f$-volumes.  Assume $(M,g)$ and $(M',g')$ isometric, then $Vol_{f}(M)=Vol_{f'}(M')$
\end{lem}

\begin{remark}
Note that in general if $f$ and $f'$ are not normalized this is not true, just take $f'=f+c$ for $c$ a nonzero constant.
\end{remark}

\begin{proof}
Since $(M,g)$ and $(M',g')$ are isometric we may view $f$ and $f'$ as soliton functions on $(M,g)$.  We see then that $\nabla^{2}(f-f')=0$ and so $f'=f+L$ where $L$ is a linear function.  If $L$=constant then $f$ and $f'$ are both normalized iff $L=0$.  So we may assume $L$ is not a constant.  Then we see that $(M,g)$ splits $\mathds{R}\times N$ such that $L(t,n)=at+b$ for $a,b$ constants.  We then also see that $f$ must restrict to a soliton function $h$ on $N$ such that $f=h+(\frac{\lambda}{2}t^{2}+a't+b')$ since the restriction of $f$ to each $\mathds{R}$ factor is quadratic.  After a change of coordinates and absorbing $b'$ into $h$ we can assume $f=h+\frac{\lambda}{2}t^{2}$.  By tracing the soliton equation $Rc+\nabla^{2}f=\lambda g$ with $\nabla f$ or with $g$ we get the equations

\[
R+|\nabla f|^{2}-2\lambda f= c = constant
\]
\[
R+\triangle f =n\lambda
\]
and by substituting we see that if $f$ is normalized then $c=0$ in the above.  Since this equation holds for both $f$ and $f'$ we see that $|\nabla L|^{2}+2<\nabla f,\nabla L> -2\lambda L = 0$ or that $L=at-\frac{1}{2\lambda}a^{2}$.  Thus we get that

\[
 \int_{M}e^{-f'}dv_{g}=\int_{\mathds{R}}e^{-(\frac{\lambda}{2}t^{2}+at-\frac{1}{2\lambda}a^{2})}dt\int_{N}e^{-h}dv_{h}
 \]
\[
 = \int_{\mathds{R}}e^{-\frac{\lambda}{2}(t-\frac{a}{\lambda})^{2}}dt\int_{N}e^{-h}dv_{h}  = \int_{\mathds{R}}e^{-\frac{\lambda}{2}(t)^{2}}dt\int_{N}e^{-h}dv_{h} = \int_{M}e^{-f}dv_{g}
\]
\end{proof}

Now the ability to alternate between viewing a soliton as a structure on a fixed Riemannian Manifold and viewing it as a Ricci flow with special properties is very convenient, especially when trying to understand limiting behavior.  With that in mind it will be useful for us to have analyzed a particular class of Ricci flows.

\begin{definition}
Let $(M,g(t))$, $t\in(-\infty,0)$, be a Ricci flow of complete Riemmanian
manifolds.  We say $(M,g(t))$ is $(C,\kappa)$-controlled if it is $\kappa$-noncollapsed and such that $|Rm[g(t)]|\leq \frac{C}{|t|}$.
\end{definition}

\begin{remark}\label{no_kappa}
It is worth noting that although the $\kappa$-noncollapsed assumption is stated for simplicity throughout, none of the estimates of this section require it.  It is only used in the final theorems to take limits.  Additionally it will be clear from the proofs that though the estimates are proved under the assumption that a global curvature bound exists on $(-\infty,0)$, if the curvature bounds exists only on intervals around $0$ or $-\infty$ then corresponding estimates exist on the respective intervals.
\end{remark}

\begin{lem}
Let $(M,g(t))$ be a $(C,\kappa)$-controlled Ricci flow.  Then there exists a sequence
$\{C^{k,l}\}$ such that $|(\frac{\partial}{\partial t})^{k}\nabla^{l}Rm|\leq \frac{C^{k,l}}{|t|^{1+k+l/2}}$.
\end{lem}
\begin{proof}
Let $(x,t)\in M\times(-\infty,0)$.  After rescaling we can assume $t=-1$.  Then on $P(x,t,1)\equiv B_{g(t)}(x,t,1)\times(t-1,t)$ we have $|Rm|\leq \frac{C}{|t|}$.  Hence standard Shi estimates as in \cite{MT} give us uniform estimates on $P(x,t,1/2)$, hence at $(x,t)$.
\end{proof}

We recall the following definitions from \cite{P} and \cite{MT}.

\begin{definition}
We call a continuous curve $\tilde{\gamma}(\tau):[0,\bar{\tau}]\rightarrow M\times(-\infty,0)$ admissible if $\exists$ $T<0$ such that $\tilde{\gamma}(\tau)=(\gamma(\tau),T-\tau)$ where $\gamma(\tau)$ is a smooth regular
curve on $(0,\bar{\tau})$.  We write $\frac{d}{d\tau}\tilde{\gamma}(\tau)=(X(\tau),-1)$,
where $X$ is the horizontal component.
\end{definition}

\begin{definition}
For any admissible curve $\tilde{\gamma}(\tau)$ we define its $\mathcal{L}$-length by $\mathcal{L}[\tilde{\gamma}(\tau)]=\int_{0}^{\bar{\tau}}\sqrt{\tau}(R(\tilde{\gamma}(\tau)) +|X|^{2})d\tau$.  Fix $(x,T)\in M\times(-\infty,0)$.  For $\forall$ $(y,T-\bar{\tau})\in M\times(-\infty,0)$, $\bar{\tau}>0$, we define the $\mathcal{L}$-distance from $(x,T)$
to $(y,T-\bar{\tau})$ as $L_{x}^{\bar{\tau}}(y) \equiv inf_{\tilde{\gamma}(\tau)}\mathcal{L}[\tilde{\gamma}(\tau)]$, where the inf is over all
admissible curves connecting $(x,T)$ to $(y,T-\bar{\tau})$.
\end{definition}

The following computation can be found in \cite{MT}

\begin{lem}
The Euler-Lagrange equation for $\mathcal{L}$ is $\nabla_{X}X-\frac{1}{2}\nabla R +\frac{1}{2\tau}X+Rc(X)=0$.
\end{lem}

It is understood the Euler-Lagrange equation is a horizontal equation for $\gamma$.

\begin{definition}
We call an admissible curve $\tilde{\gamma}$ which satisfies the Euler-Lagrange equation
an $\mathcal{L}$-geodesic.
\end{definition}

\begin{lem}
Let $(M,g(t))$ be a $(C,\kappa)$-controlled Ricci flow.  Fix $(x,T)\in M\times(-\infty,0)$,
$\bar{\tau}>0$.  Then $\forall$ $y\in M$ there exists a minimizing $\mathcal{L}$-geodesic from
$(x,T)$ to $(y,T-\bar{\tau})$.
\end{lem}

\begin{proof}
Note $|Rm|$ bounded uniformly on $M\times [T-\bar{\tau},T]$.  Hence the result follows
as in \cite{MT}.
\end{proof}

From the above it follows (see \cite{MT}) that $L_{x}^{\tau}$ is a locally lipschitz function and for each $\tau>0$  $\exists$ an open dense subset $U_{x}^{\tau}\subseteq M$ such that $L_{x}^{\tau}$ is smooth and $\exists$ a unique minimizing $\mathcal{L}$-geodesic to each
point in $U_{x}^{\tau}$.  It holds that $\forall q\in U_{x}^{\tau}$ that $\nabla L_{x}^{\tau}=2\sqrt{\tau}X$, where $X$ is the horizontal tangent of the unique $\mathcal{L}$-geodesic to $q$.

\begin{definition}
We define the reduced length function $l^{\tau}_{x}=\frac{L_{x}^{\tau}}{2 \sqrt{\tau}}$.  For an admissible curve $\tilde{\gamma}:[0,\bar{\tau}]\rightarrow M\times (-\infty,0)$ we define $\mathcal{K}^{\tau}[\tilde{\gamma}]= \int_{0}^{\bar{\tau}}\tau^{3/2}H(X)d\tau$ where $H(X)=-R_{\tau}-\frac{R}{\tau}-2<\nabla R,X>+2Rc(X,X)$ is the Harnack functional.
\end{definition}

\begin{remark}
The importance of normalizing the $\mathcal{L}$-length is in the scale invariance of $l_{x}^{\tau}$.  Let $c>0$ and note that $c^{-1}g(T+c t)$ is also a Ricci flow on $M$.  Then if $l_{x}^{',\tau}$ is the reduced length function for the rescaled Ricci flow we observe that $l_{x}^{\tau} = l_{x}^{',\tau/c}$.
\end{remark}

We use the following tools (\cite{MT} or \cite{P}):

\begin{propA}
For $\forall$ $q\in U_{x}^{\tau}$ let $\tilde{\gamma}_{q}$ the unique minimizing $\mathcal{L}$-geodesic from $(x,T)$ to $(q,T-\tau)$.  Then at $(q,T-\tau)$

$1)$ $\frac{\partial l^{\tau}_{x}}{\partial\tau} = R(q,\tau)-\frac{l^{\tau}_{x}}{\tau} +\frac{1}{2\tau^{3/2}}\mathcal{K}^{\tau}[\tilde{\gamma}_{q}]$

$2)$ $|\nabla l^{\tau}_{x}|^{2} = \frac{l^{\tau}_{x}}{\tau} - R(q,\tau) - \frac{1}{\tau^{3/2}}\mathcal{K}^{\tau}[\tilde{\gamma}_{q}]$

$3)$ $\triangle l^{\tau}_{x}(q) \leq \frac{n}{2\tau} - R(q,\tau) - \frac{1}{2\tau^{3/2}}\mathcal{K}^{\tau}[\tilde{\gamma}_{q}]$
\end{propA}
\begin{remark}
The proofs of (1) and (2) are purely computational.  (3) involves an estimate on the second variation formula for $\mathcal{L}_{x}^{\tau}$, not unlike the proving of the Laplace comparison theorems.
\end{remark}

Rewriting the above we get

\begin{propB}
For $\forall$ $q\in U_{x}^{\tau}$

$b1)$ $\frac{\partial l^{\tau}_{x}}{\partial\tau} - \triangle l^{\tau}_{x}(q) + |\nabla l^{\tau}_{x}|^{2} - R(q,\tau) + \frac{n}{2\tau} = \delta \geq 0$

$b2)$ $2\triangle l^{\tau}_{x}(q) - |\nabla l^{\tau}_{x}|^{2} + R(q,\tau) + \frac{l^{\tau}_{x} - n}{\tau} = -2\delta \leq 0$

In fact, $(b1)$,$(b2)$ hold globally in the distributional sense and $(b1)=-2(b2)$ as
distributions.
\end{propB}

With the above we begin to analyze the reduce length functions on a controlled Ricci flow.

\begin{prop}\label{l_m}
Let $(M,g(t))$ be a $(C,\kappa)$-controlled Ricci flow and $(x,T)\in M\times (-\infty,0)$. Let $l^{\tau}_{x}$ be the reduced length function.  Then there exists $m=m(n,C)$ such that

$1)$ $l^{\tau}_{x}(y)\geq -m$ $\forall y\in M$

$2)$ $|l^{\tau}_{x}(x)| \leq m$
\end{prop}

\begin{proof}
Let $\gamma_{y}$ be a minimizing $\mathcal{L}$-geodesic from $(x,T)$ to $(y,T-\bar{\tau})$.

Then $\mathcal{L}[\gamma_{y}]= \int_{0}^{\bar{\tau}}\sqrt{\tau}(R(\tilde{\gamma}(\tau)) +|X|^{2})d\tau$.  Now $|R|(y,\tau)\leq \frac{\tilde{C}}{|T-\tau|}\leq \frac{\tilde{C}}{\tau}$ with $\tilde{C}=\tilde{C}(n,C)$ and hence $\mathcal{L}[\gamma_{y}]\geq -\tilde{C}\int_{0}^{\bar{\tau}}\tau^{-1/2}d\tau=-2\tilde{C}\sqrt{\bar{\tau}}$.

If $y=x$ we can let $\sigma$ be the constant path to see $L_{x}^{\tau} \leq \mathcal{L}[\sigma]= \int_{0}^{\bar{\tau}}\sqrt{\tau}R \leq 2\tilde{C}\sqrt{\bar{\tau}}$
\end{proof}

To get growth estimates on the reduced length functional we will show the following

\begin{lem}
Let $(M,g(t))$ be a $(C,\kappa)$-controlled Ricci flow and $(x,T)\in M\times (-\infty,0)$. Let $l^{\tau}_{x}$ be the reduced length function.  Let $(y,T-\bar{\tau})\in M\times(-\infty,0)$, $\bar{\tau}>0$, and let $\gamma_{y}$ be a minimizing $\mathcal{L}$-geodesic from $(x,T)$ to $(y,T-\bar{\tau})$.  Then there exists $A=A(n,C)$ such that $|\mathcal{K}^{\bar{\tau}}[\gamma_{y}]| \leq A\sqrt{\tau}(1+m+l_{x}^{\bar{\tau}}(y))$
\end{lem}

\begin{proof}
As before $\exists$ $\tilde{C}=\tilde{C}(n,C)$ such that $|R|(y,\tau)\leq \frac{\tilde{C}}{|T-\tau|}\leq \frac{\tilde{C}}{\tau}$, $|\nabla R| \leq \frac{\tilde{C}}{\tau^{3/2}}$ and $|\frac{\partial R}{\partial\tau}|\leq \frac{\tilde{C}}{\tau^{2}}$.  We will use $\tilde{C}$ and $\tilde{C}'$ to denote a constant
depending on only $n,C$, though $\tilde{C}$ itself may change from line to line.

So we have

\[
\mathcal{K}^{\bar{\tau}}[\gamma_{y}] = \int_{0}^{\bar{\tau}} \tau^{3/2} (-R_{\tau}-\frac{R}{\tau}-2<\nabla R,X>+2Rc(X,X))d\tau
\]

\[
\leq \tilde{C}\int_{0}^{\bar{\tau}}\tau^{-1/2} + 2\int_{0}^{\bar{\tau}}\tau^{3/2}|\nabla R||X| + \tilde{C}\int_{0}^{\bar{\tau}}\sqrt{\tau}|X|^{2}
\]

\[
\leq \tilde{C}\sqrt{\bar{\tau}} + \int_{0}^{\bar{\tau}}\tau^{3/2}(\tau|\nabla R|^{2}+\frac{|X|^{2}}{\tau}) + \tilde{C}\int_{0}^{\bar{\tau}}\sqrt{\tau}(R+|X|^{2})- \tilde{C}\int_{0}^{\bar{\tau}}\sqrt{\tau}R
\]

\[
\leq \tilde{C}\sqrt{\bar{\tau}} + \tilde{C}'\sqrt{\bar{\tau}}l_{x}^{\bar{\tau}}(y) \leq \tilde{C}\sqrt{\bar{\tau}} + \tilde{C}'\sqrt{\bar{\tau}}(l_{x}^{\bar{\tau}}(y) + m)
\]

\[
\leq A\sqrt{\bar{\tau}}(1 + m + l_{x}^{\bar{\tau}}(y))
\]
\end{proof}

The following proposition is key to controlled the reduced length function.  Under various curvature assumptions similar estimates may be found in \cite{P},\cite{MT},\cite{En}.

\begin{prop}
There exists $A=A(n,C)$ such that $\forall y\in M$

$1)$ $l_{x}^{\tau}(y) \leq A(1+\frac{d_{g(T-\tau)}(x,y)}{\sqrt{\tau}})^{2}$

$2)$ $|\nabla l_{x}^{\tau}|(y) \leq \frac{A}{\sqrt{\tau}} (1+\frac{d_{g(T-\tau)}(x,y)}{\sqrt{\tau}})$

$3)$ $|\frac{\partial l_{x}^{\tau}}{\partial\tau}|(y) \leq \frac{A}{\tau}(1+\frac{d_{g(T-\tau)}(x,y)}{\sqrt{\tau}})^{2}$
\end{prop}

\begin{proof}
Let $q\in U_{x}^{\tau}$.  Then

\[
|\nabla l_{x}^{\tau}|^{2}(q)= \frac{l^{\tau}_{x}}{\tau} - R(q,\tau) - \frac{1}{\tau^{3/2}}\mathcal{K}^{\tau}[\tilde{\gamma}_{q}]
\]

\[
\leq \frac{A}{\tau}(1+m+l_{x}^{\tau}(q))
\]

 for some $A=A(n,c)$.  Since $l_{x}^{\tau}$ is Lipschitz this must hold on all $M$ in Lipschitz sense.  Let $\sigma(s):[0,d_{g(T-\tau)}]$ be a minimizing geodesic from $x$ to $y$ in $M\times\{T-\tau\}$.  Let $z(s) = l_{x}^{\tau}(\sigma(s))$.  Then $z(s)$ is Lipschitz and
$z(s)>-m$ by Proposition (\ref{l_m}).  Let $h(s)$ be a solution of $\dot{h}=\sqrt{\frac{A}{\tau}(1+m+h(s))}$ with $h(0) = l_{x}^{\tau}(x)>-m$ (so the solution exists, is unique and is smooth).

Claim:  $z(s)\leq h(s)$ $\forall s\in [0,d]$

Proof of Claim:

This is straightforward, the main point is to stay away from the singular initial condition $-(m+1)$.  Since $z(s)>-m$, $|\dot{z}|\leq \sqrt{\frac{A}{\tau}(1+m+z(s))}$.  Let $h_{\epsilon}(s)$ be a solution with $h_{\epsilon}(0)=h(0)+\epsilon$ (so also smooth and unique).  If $z(s)$ is not less than $h_{\epsilon}(s)$ then $\exists$ a first $s'\in (0,d]$ such that $z(s')=h_{\epsilon}(s')$.  But for $s<s'$ $|\dot{z}|\leq \sqrt{\frac{A}{\tau}(1+m+z(s))} < \sqrt{\frac{A}{\tau}(1+m+h_{\epsilon}(s))}=h_{\epsilon}(s)$, which is a contradiction to $z(s')=h_{\epsilon}(s')$.  Hence $z(s)<h_{\epsilon}(s)$ $\forall s\in [0,d]$ and by limiting we get our result. $\circ$

Now it is easy to check that for some $a_{i}=a_{i}(m,A)$ $h(s)= a_{2}(\frac{s}{\sqrt{\tau}})^{2} + a_{1}(\frac{s}{\sqrt{\tau}}) + a_{0}$.  Hence we get
that $l_{x}^{\tau}(y) \leq A(1+\frac{d_{g(T-\tau)}(x,y)}{\sqrt{\tau}})^{2}$ for some $A=A(n,C)$.  By plugging into $(a1)$, $(a2)$ of proposition A and using our bound from the last lemma we get our result.
\end{proof}

We will need the following to get good lower bound estimates on $l_{x}^{\tau}$:

\begin{lem}
Let $(M,g(t))$ be a $(C,\kappa)$-controlled Ricci flow.  Let $\sigma(\eta)$ be a unit speed minimizing geodesic in $(M,g(T-\tau))$.  Then there exists $A=A(n,C)$ such that $\int_{\sigma}Rc(\dot{\sigma},\dot{\sigma}) \leq \frac{A}{\sqrt{|T-\tau|}}$
\end{lem}

\begin{proof}
Let $s=|T-\tau|$.  If the length of $\sigma$, $|\sigma|=d$, is less than $2\sqrt{s}$, then
we see

\[
\int_{\sigma}Rc(\dot{\sigma},\dot{\sigma}) \leq \int_{0}^{2\sqrt{s}}\frac{\tilde{C}}{s}d\eta \leq \frac{A}{s}\sqrt{s}=\frac{A}{\sqrt{s}}
\]

as claimed.  Hence we can assume the length is larger than $2\sqrt{s}$.

Let $\{E^{i}\}$ be an orthonormal basis at $\sigma(0)$,  $\{E^{i}(\eta)\}$ the parallel translation along $\sigma$.  Let $\{Y^{i}(\eta)\}=h(\eta)E^{i}(\eta)$ for some piecewise smooth $h$ such that $h(0)=h(d)=0$.  Since $\sigma$ is minimizing we can take the second variation of the (Riemannian) energy in the direction $Y^{i}$ to get

\[
0\leq \delta^{2}_{Y,Y}E = \int_{\sigma}|\nabla_{\dot{\sigma}}Y^{i}|^{2} -<R(Y^{i},\dot{\sigma})\dot{\sigma},Y^{i}>
\]

\[
= \int_{\sigma}(h')^{2}-h^{2}<R(E^{i},\dot{\sigma})\dot{\sigma},E^{i}>
\].

Summing over $i$ yields

\[
0\leq (n-1)\int_{\sigma}(h')^{2} - \int_{\sigma} h^{2}Rc(\dot{\sigma}, \dot{\sigma})
\]

or

\[
\int_{\sigma}Rc(\dot{\sigma}, \dot{\sigma}) \leq (n-1)\int_{\sigma}(h')^{2} + \int_{\sigma} (1-h^{2})Rc(\dot{\sigma}, \dot{\sigma})
\]

Let $$h(\eta) = \left\{\begin{array}{lr}
\frac{\eta}{\sqrt{s}} &   0\leq \eta  \leq \sqrt{s}\\
1 &   \sqrt{s}\leq \eta  \leq d-\sqrt{s}\\
\frac{d-\eta}{\sqrt{s}} & d-\sqrt{s}\leq \eta\leq d
\end{array}\right.
 $$.  Then

\[
\int_{\sigma}Rc(\dot{\sigma}, \dot{\sigma}) \leq (n-1)(\frac{1}{s}2\sqrt{s})+ \int_{0}^{\sqrt{s}}(1-\frac{\eta^{2}}{s})Rc(\dot{\sigma}, \dot{\sigma}) + \int_{0}^{\sqrt{s}}(1-\frac{\eta^{2}}{s})Rc(\dot{\sigma}, \dot{\sigma})
\]

\[
\leq \frac{2(n-1)}{\sqrt{s}}+\frac{2\tilde{C}}{s}(\sqrt{s}-\frac{1}{3s}s^{3/2}) = \frac{A}{\sqrt{s}}
\]
\end{proof}

\begin{prop}
Let $(M,g(t))$ be a $(C,\kappa)$-controlled Ricci flow and $(x,T)\in M\times (-\infty,0)$.  Then there exists $A=A(n,C)$ such that $l_{x}^{\tau}(y)\geq \frac{1}{A}(1+\frac{d_{g(T-\tau)}(x,y)}{\sqrt{\tau}})^{2}-A$
\end{prop}

\begin{proof}
Let $\gamma_{x}(s)$, $\gamma_{y}(s)$ be minimizing $\mathcal{L}$-geodesics from $(x,T)$ to $(x,T-\tau)$,$(y,T-\tau)$, respectively.

Define $h(s)=d_{g(T-s)}(\gamma_{x}(s),\gamma_{x}(s))$, a locally lipschitz function.  Then in the sense of forward difference quotients we have

\[
\dot{h}(s) = <\nabla_{1}d_{s}(\gamma_{x}(s),\gamma_{x}(s)),\nabla l_{x}^{s}(\gamma_{x}(s))> + <\nabla_{2}d_{s}(\gamma_{x}(s),\gamma_{x}(s)),\nabla l_{x}^{s}(\gamma_{y}(s))> + \frac{\partial}{\partial s}d_{s}(\gamma_{x}(s),\gamma_{x}(s))
\]

\[
\leq |\nabla l_{x}^{s}|(\gamma_{x}(s)) + |\nabla l_{x}^{s}|(\gamma_{y}(s)) + \frac{\partial}{\partial s}\int_{\sigma}\sqrt{g_{s}(\dot{\sigma},\dot{\sigma})}
\]

where $\sigma$ is a unit speed minimizing geodesic in $(M,g(T-s))$ connecting $\gamma_{x}(s)$ to $\gamma_{y}(s)$ (note that $d_{s'}(\gamma_{x}(s'),\gamma_{x}(s'))\leq \int_{\sigma}\sqrt{g_{s'}(\dot{\sigma},\dot{\sigma})}$ for all $s'$ and equal at $s$.  Hence inequality holds in forward difference sense).  But by the last lemma

\[
\frac{\partial}{\partial s}\int_{\sigma}\sqrt{g_{s}(\dot{\sigma},\dot{\sigma})} = \int_{\sigma}Rc(\dot{\sigma},\dot{\sigma}) \leq \frac{A}{\sqrt{s}}
\].

So we have that

\[\dot{h}(s) \leq \sqrt{\frac{A}{s}(1+m+l_{x}^{s}(\gamma_{x}(s)))} + \sqrt{\frac{A}{s}(1+m+l_{x}^{s}(\gamma_{y}(s)))} + \frac{A}{\sqrt{s}}
\].

But if $\gamma$ is a minimizing $\mathcal{L}$-geodesic along $[0,\tau]$ then we have that

\[
l_{x}^{s}(\gamma(s))=\frac{\int_{0}^{s}\sqrt{\tau(R+|X|^{2})}}{2\sqrt{s}} \leq
\frac{\int_{0}^{\bar{\tau}}\sqrt{\tau(R+|X|^{2})} - \int_{s}^{\bar{\tau}}\sqrt{\tau}R}{2\sqrt{s}}
\]

\[
\leq \frac{\sqrt{\bar{\tau}}}{\sqrt{s}}l_{x}^{\tau}(\gamma(\tau)) + \frac{\tilde{C}(\sqrt{\tau}-\sqrt{s})}{\sqrt{s}} \leq \frac{\sqrt{\bar{\tau}}}{\sqrt{s}}(A+l_{x}^{\tau}(\gamma(\tau)))
\].

Hence after possibly increasing $A$ we have

\[
\dot{h}(s) \leq \frac{A\tau^{1/4}}{s^{3/4}}(1+\sqrt{1+m+l_{x}^{\tau}(y)}) +\frac{A}{\sqrt{s}}
\]

\[\Rightarrow d_{T-\tau}(x,y)=h(\tau) \leq A\sqrt{\tau}(1+\sqrt{(1+m+l_{x}^{\tau}(y))})\]
\end{proof}

We sum up by

\begin{propC}
Let $(M,g(t))$ be a $(C,\kappa)$-controlled Ricci flow and $(x,T)\in M\times (-\infty,0)$.  Then there exists $A=A(n,C)$ and $m=m(n,C)$ such that

$1)$ $|l_{x}^{\tau}(x)|\leq m$

$2)$ $\frac{1}{A}(1+\frac{d_{g(T-\tau)}(x,y)}{\sqrt{\tau}})^{2} - A \leq l_{x}^{\tau}(y) \leq A(1+\frac{d_{g(T-\tau)}(x,y)}{\sqrt{\tau}})^{2}$

$3)$ $|\nabla l_{x}^{\tau}|(y) \leq \frac{A}{\sqrt{\tau}} (1+\frac{d_{g(T-\tau)}(x,y)}{\sqrt{\tau}})$

$4)$ $|\frac{\partial l_{x}^{\tau}}{\partial\tau}|(y) \leq \frac{A}{\tau}(1+\frac{d_{g(T-\tau)}(x,y)}{\sqrt{\tau}})^{2}$
\end{propC}

As by Perelman we introduce

\begin{definition}
Let $A\subseteq M\times(-\infty,0)$ be a measurable subset.  Let $(x,T)\in M\times(-\infty,0)$ be fixed.  Define $A_{\tau}=A\cap(M\times\{T-\tau\})$.  We define the reduced volume of $A$ at $\tau$ by $\mathcal{V}_{A}(\tau)=\int_{A_{\tau}}\tau^{-n/2}e^{-l_{x}^{\tau}}dv_{g(T-\tau)}$.  If $A_{\tau}=M$ $\forall \tau$ we simply call $\mathcal{V}_{M}(\tau)$ the reduced volume.
\end{definition}

\begin{remark}
If it is important to distinguish the Ricci flow on $M$ we will write $\mathcal{V}_{(M,g(T-\tau))}(\tau)$.  An important property of the reduced volume is a scale invariance.  Let $c>0$ and note that $c^{-1}g(T+c t)$ is also a Ricci flow on $M$.  Then we observe that $\mathcal{V}_{(M,g(T-\tau))}(\tau) = \mathcal{V}_{(M,c^{-1}g(T-\tau/c))}(\tau/ c)$.
\end{remark}

\begin{lem}
Let $(M,g(t))$ be a $(C,\kappa)$-controlled Ricci flow and $(x,T)\in M\times (-\infty,0)$.  Then $(\frac{\partial l_{x}^{\tau}}{\partial\tau}+\frac{n}{2\tau}-R)e^{-l_{x}^{\tau}}\in L^{1}(M,dv_{g(T-\tau)})$ and $\int_{M}\tau^{-n/2}e^{-l_{x}^{\tau}}(\frac{\partial l_{x}^{\tau}}{\partial\tau}+\frac{n}{2\tau}-R)\geq 0$
\end{lem}

\begin{proof}
We have quadratic bounds on $|\frac{\partial l_{x}^{\tau}}{\partial\tau}|$ as well
as quadratic lower bounds on $l_{x}^{\tau}$. Since the curvature is bounded the volume
form grows at most exponentially in normal coordinates and hence  $(\frac{\partial l_{x}^{\tau}}{\partial\tau}+\frac{n}{2\tau}-R)e^{-l_{x}^{\tau}}$ is integrable.

For $\forall \phi\in C^{\infty}_{c}(M)$, $\phi\geq 0$ we have by proposition (B)

\[
\int_{M}\phi(\frac{\partial l_{x}^{\tau}}{\partial\tau}+|\nabla l_{x}^{\tau}|^{2}+\frac{n}{2\tau}-R)-l_{x}^{\tau}\triangle\phi dv_{g} \geq 0
\].

Since $l_{x}^{\tau}$ is Lipschitz we may write this as

\[
\int_{M}\phi(\frac{\partial l_{x}^{\tau}}{\partial\tau}+|\nabla l_{x}^{\tau}|^{2}+\frac{n}{2\tau}-R)+<\nabla l_{x}^{\tau},\nabla\phi> dv_{g} \geq 0
\].

By limiting we see this inequality holds for all $\phi\in C^{0,1}_{c}$, $\phi\geq 0$.  Let $\phi_{r}$ be a smooth cutoff with $\phi_{r}=1$ in $B_{r}(x)$, $\phi_{r}=0$ outside $B_{2r}(x)$ and $|\phi_{r}|\leq \frac{3}{r}$.

Let $\phi=\tau^{-n/2}e^{-l_{x}^{\tau}}\phi_{r}$.  Then we see $\forall r>0$ that

\[
\int_{M}\tau^{-n/2}e^{-l_{x}^{\tau}}[\phi_{r}(\frac{\partial l_{x}^{\tau}}{\partial\tau}+|\nabla l_{x}^{\tau}|^{2}+\frac{n}{2\tau}-R)+<\nabla l_{x}^{\tau},\nabla\phi_{r}>-\phi_{r}|\nabla l_{x}^{\tau}|^{2}] \geq 0
\].

\[
\Rightarrow \int_{M}\tau^{-n/2}e^{-l_{x}^{\tau}}[\phi_{r}(\frac{\partial l_{x}^{\tau}}{\partial\tau}+\frac{n}{2\tau}-R)+<\nabla l_{x}^{\tau},\nabla\phi_{r}>] \geq 0
\]

Again by our estimates we see $e^{-l_{x}^{\tau}}|\nabla l_{x}^{\tau}|$, $(\frac{\partial l_{x}^{\tau}}{\partial\tau}+\frac{n}{2\tau}-R)e^{-l_{x}^{\tau}}\in L^{1}(M,dv_{g(T-\tau)})$. Hence we can limit out $r$ to get our result.
\end{proof}

\begin{remark}\label{dist_ncom}
A useful corollary of the above proof is that the distribution $\mathcal{D}[\phi] = \int_{M}\phi(\frac{\partial l_{x}^{\tau}}{\partial\tau}+|\nabla l_{x}^{\tau}|^{2}+\frac{n}{2\tau}-R)+<\nabla l_{x}^{\tau},\nabla\phi> dv_{g}$ extends to a nonnegative continuous linear functional on the space noncompactly supported functions of the form $\phi = \psi e^{-l_{x}^{\tau}}$, where $\psi$ is a smooth function with uniform bounds on $|\psi|$ and $|\nabla\psi|$.  Extended in this form a proof identical to the one above shows that $\int(\triangle l_{x}^{\tau}-|\nabla l_{x}^{\tau}|^{2}) e^{-l_{x}^{\tau}} = 0$.
\end{remark}

\begin{prop}
Let $(M,g(t))$ be a $(C,\kappa)$-controlled Ricci flow.  Then $\frac{d}{d\tau}\mathcal{V}_{M}(\tau)\leq 0$ and $\mathcal{V}_{M}(\tau)\leq (4\pi)^{n/2}$
\end{prop}
\begin{proof}
Because $l_{x}^{\tau}$ is Lipschitz and the below quantities are $L^{1}$ we can write

\[
\frac{d}{d\tau}\mathcal{V}_{M}(\tau) = \frac{\partial}{\partial\tau}(\int_{M}\tau^{n/2}e^{-l_{x}^{\tau}}dv_{g(T-\tau)})
\]

\[
= - \int_{M}(\frac{\partial l_{x}^{\tau}}{\partial\tau}+\frac{n}{2\tau}-R)\tau^{n/2}e^{-l_{x}^{\tau}}dv_{g(T-\tau)} \leq 0
\]

To see $\mathcal{V}_{M}(\tau)\leq (4\pi)^{n/2}$ let $\tau_{i}\rightarrow 0$ with $g_{i}(t)=\tau_{i}^{-1}g(T+\tau_{i}t)$ and $l_{i}^{\tau}=l_{x}^{\tau/\tau_{i}}$ the reduced length with respect to the rescaled Ricci flow.  As a Ricci flow we see $(M,g_{i}(t)),(x,-1))\rightarrow (\mathds{R}^{n},g_{0}(t),(0,-1))$ with $l_{i}^{\tau}\rightarrow \frac{1}{2\sqrt{\tau}}|x|^{2}$ becoming quadratic.  Because of our decay estimates on $l_{x}^{\tau}$ and our uniform curvature bounds we see that $\mathcal{V}_{(M,g_{i})}(-1)$ is converging to the reduced volume $\mathcal{V}_{(\mathds{R}^{n},g_{0})}(-1)$, which by a computation we can see is $(4\pi)^{n/2}$.  Since $\mathcal{V}_{M}$ is monotone  and $\mathcal{V}_{(M,g_{i})}(-1) = \mathcal{V}_{(M,g)}(\tau_{i})$ we have our result.
\end{proof}

In our analysis it will not quite be enough to study the reduced length functions of points in our space-time.  It will be useful as a tool to have a function which behaves as a reduced length function but is globally defined on $M\times(-\infty,0)$.  Intuitively we will construct a reduced length function from a singular point on the boundary of our space time:

\begin{prop}
Let $(M,g(t))$ be a $(C,\kappa)$-controlled Ricci flow, $x\in M$ and $T_{i}\rightarrow 0 \in (-\infty,0)$.  Let $l_{i}^{\tau}=l_{(x,T_{i})}^{\tau}$ the reduced length functions to $(x,T_{i})$.  Then after possibly passing to a subsequence $\exists$ $\bar{l}^{\tau}\in C^{0,1}(M\times(-\infty,0))$ such that $l_{i}^{\tau}\rightarrow \bar{l}^{\tau}$ in $C^{0,\alpha}_{loc}$ and weakly in $W^{1,2}_{loc}$ with Propositions (B) and (C) holding for $\bar{l}^{\tau}$.
\end{prop}

\begin{proof}
We note that Proposition (C) does not depend on $T_{i}$.  Hence uniform $C^{0,1}$ bounds on $l_{i}^{\tau}$ on compact subsets of $M\times(-\infty,T_{i})$ imply the existence of $\bar{l}^{\tau}\in C^{0,1}(M\times(-\infty,0))$ such that $l_{i}^{\tau}\rightarrow \bar{l}^{\tau}$ in $C^{0,\alpha}_{loc}$ and weakly in $W^{1,2}_{loc}$.  That Proposition (C) holds for $\bar{l}^{\tau}$ is immediate, we just need to check (B).  To check (B) it is enough to check that $\forall \phi\in C^{\infty}_{c}$ that $\int\phi|\nabla \bar{l}^{\tau}|^{2} = lim \int\phi|\nabla l^{\tau}_{i}|^{2}$, since the other terms in (B) clearly converge in the required way.  The proof for this is as in \cite{MT}.
\end{proof}

\begin{definition}
We call $\bar{l}^{\tau}$ from the last proposition a singular reduced length function from $(x,0)$.  The last proposition implies that we can define a singular reduced volume $\mathcal{\bar{V}}_{M}(\tau) =\int_{M}\tau^{-n/2}e^{-\bar{l}^{\tau}}dv_{g}$.
\end{definition}

\begin{lem}
$\frac{d}{d\tau}\mathcal{\bar{V}}_{M}(\tau)\leq 0$ and $\mathcal{\bar{V}}_{M}(\tau)\leq (4\pi)^{n/2}$
\end{lem}

\begin{proof}

Because (B) and (C) hold for $\bar{l}^{\tau}$ the proof that $\mathcal{\bar{V}}_{M}$ is monotone is the same as before.  To see the upper bound still holds we note that by the growth estimates of (C) and our uniform curvature bounds we have that for each fixed $\tau$ that $\int_{M}\tau^{-n/2}e^{-l_{i}^{\tau}}dv_{g}\rightarrow \int_{M}\tau^{-n/2}e^{-\bar{l}^{\tau}}dv_{g}$, and hence the upper bounds of the reduced volumes for $l^{\tau}_{i}$ imply upper bounds for the reduced volumes of $\bar{l}^{\tau}$.

\end{proof}

We will be exploiting the following theorem later:

\begin{thm}\label{asym_sol}
Let $(M,g(t))$ be a $(C,\kappa)$-controlled Ricci flow,  $x\in M$.  Let $\tau^{-}_{i}\rightarrow 0$ and $\tau^{+}_{i} \rightarrow \infty$ with $g^{\pm}_{i}(t)=(\tau^{\pm}_{i})^{-1}g(\tau^{\pm}_{i}t)$.  Then, after possibly passing to subsequences, $(M,g^{\pm}_{i}(t),(x,-1)) \rightarrow (S^{\pm},h^{\pm}(t),(x^{\pm},-1))$ where $(S^{\pm},h^{\pm}(t),x^{\pm})$ are $(C,\kappa)$-controlled shrinking solitons which are normalized at $t=-1$.  Further, if $S^{\pm}$ are isometric then $(M,g(t))$ is also isometric to $S^{\pm}$, and hence is a shrinking soliton.
\end{thm}
\begin{proof}
Let $\bar{l}^{\tau}$ be a singular reduced length for $(x,0)$ with reduced volume $\mathcal{\bar{V}}_{M}$,and let $\bar{l}^{\pm,\tau}_{i} = \bar{l}^{\tau/\tau^{\pm}_{i}}$ be singular reduced lengths for $(M,g^{\pm}_{i}(t))$ at $(x,0)$ with respective reduced volumes $\mathcal{\bar{V}}^{\pm}_{i}$.  By compactness we get the existence of some $(C,\kappa)$-controlled Ricci flows $(S^{\pm},h^{\pm}(t),x^{\pm})$ after passing to subsequences.  As in the construction of $\bar{l}^{\tau}$ we can limit out $\bar{l}^{+,\tau}_{i}$, $\bar{l}^{-,\tau}_{i}$ to $C^{0,1}$ functions $\bar{l}^{\pm,\tau}$ on $S^{\pm}\times(-\infty,0)$ which satisfy propositions (B) and (C), which have respective reduced volumes $\mathcal{\bar{V}}^{\pm}(\tau)$.

Now, however, there is some additional structure.  Recall that by scale invariance $\mathcal{\bar{V}}^{\pm}_{i}(\tau) =\mathcal{\bar{V}}(\tau\tau^{\pm}_{i})$, and that from the decay estimates in (C) we have $\mathcal{\bar{V}}^{\pm}_{i}(\tau)\rightarrow \mathcal{\bar{V}}^{\pm}(\tau)$.  But both $\mathcal{\bar{V}}(\tau\tau^{\pm}_{i})$ are bounded monotone sequence, and hence converge regardless of $\tau$.  We see that $\mathcal{\bar{V}}^{\pm}(\tau)=c^{\pm}=constants$.  But as in \cite{MT} we see then that $0 = \frac{d}{d\tau}\mathcal{\bar{V}}^{\pm}(\tau) = -\int_{M}(\frac{\partial l^{\pm}}{\partial\tau} - R(q,\tau) + \frac{n}{2\tau})\tau^{-n/2}e^{-l^{\pm}}dv_{g^{\pm}} = -\int_{M}(\frac{\partial l^{\pm}}{\partial\tau} - \triangle l^{\pm}(q) + |\nabla l^{\pm}|^{2} - R(q,\tau) + \frac{n}{2\tau})\tau^{-n/2}e^{-l^{\pm}}dv_{g^{\pm}}$, where it is understood $\triangle l^{\pm}$ has been extended as in remark (\ref{dist_ncom}).  For simplicity we write $D = \frac{\partial l^{\pm}}{\partial\tau} - \triangle l^{\pm}(q) + |\nabla l^{\pm}|^{2} - R(q,\tau) + \frac{n}{2\tau}$ to be this linear functional.  But then $\forall \phi$ with compact support and $0\leq \phi \leq 1$ we can write $0 = \int_{M}D\tau^{-n/2}e^{-l^{\pm}}dv_{g^{\pm}} = \int_{M}D\phi\tau^{-n/2}e^{-l^{\pm}}dv_{g^{\pm}} + \int_{M}D(1-\phi)\tau^{-n/2}e^{-l^{\pm}}dv_{g^{\pm}}\geq \int_{M}D\phi\tau^{-n/2}e^{-l^{\pm}}dv_{g^{\pm}} \geq 0$.  By splitting any $|\phi|\leq 1$ into positive and negative parts we see $D=0$ as a distribution.  Since $l^{\pm}$ are lipschitz it follows from standard estimates that $l^{\pm}$ are in fact smooth solutions to $\frac{\partial l^{\pm}}{\partial\tau} - \triangle l^{\pm}(q) + |\nabla l^{\pm}|^{2} - R(q,\tau) + \frac{n}{2\tau} = 0$ and hence $2\triangle l^{\pm}(q) - |\nabla l^{\pm}|^{2} + R(q,\tau) + \frac{l^{\pm} - n}{\tau}=0$.  It is then a clever computation of Perelman's (\cite{P},\cite{MT}) that this implies that $l^{\pm}(\tau)$ are soliton functions with soliton constants $\frac{1}{2\tau}$.  Since $\l^{\pm}$ are solitons we see $R+\triangle l^{\pm} = n\frac{1}{2\tau}$ and $R+|\nabla l^{\pm}|^{2}-2\lambda l^{\pm} = a^{\pm}=constants$.  From $(2\triangle l^{\tau}_{x}(q) - |\nabla l^{\tau}_{x}|^{2} + R(q,\tau) + \frac{l^{\tau}_{x} - n}{\tau}=0)$ we get though that $a^{\pm}=0$ and hence $l^{\pm}$ are normalized.  It then follows that if $(S^{\pm},g^{\pm})$ are isometric then $c^{+}=c^{-}$.  But $lim_{\tau\rightarrow 0}\mathcal{\bar{V}}(\tau)=c^{+}$ and $lim_{\tau\rightarrow \infty}\mathcal{\bar{V}}(\tau)=c^{-}$, with $\mathcal{\bar{V}}(\tau)$ monotone. Hence $\mathcal{\bar{V}}(\tau)=constant$ and by the same arguments $\bar{l}^{\tau}$ is a shrinking soliton structure on $(M,g(t))$ which must be isometric to $(S^{\pm},g^{\pm})$.
\end{proof}

\begin{remark}
A little care is needed in the above.  If we know $(M,g(t))$ is a shrinking soliton it is not necessarily the case that the asymptotic soliton $(S^{-},g^{-})$ is isometric to $(M,g)$, and also not necessarily the case that the singular reduced length function $\bar{l}^{\tau}$ from $(x,0)$ is a soliton function.  As we will see it is, however, always the case that $(M,g)$ is isometric to $(S^{+},g^{+})$, a phenomena which will have applications for studying the collapseness and gradient behavior of an arbitrary shrinking soliton.
\end{remark}

We apply the tools from this section to finish the proof of Theorem \ref{mthm5}

\begin{proof}[Theorem \ref{mthm5}]
As in remark (\ref{no_kappa}) we point out that all the estimates of this section hold on $[-T,0)$ if we only assume a curvature estimate $|Rm|\leq \frac{C}{|t|}$ for $t\in [-T,0)$.  So we may still construct a singular reduced length function at $(x,0)$ and if $(M,g(t))$ is uniformly $\kappa$-noncollapsed at $(x,t)$ then verbatim as the last theorem we may limit out $(M,g_{i}(t),(x,-1))$ and the singular reduced length functions $l_{i}$ to a normalized soliton function $l$ on the limit.  That $(M,g(t))$ is $\kappa$-noncollapsed at $(x,t)$ follows by a result of Perelman's (Theorem \ref{vol_noncol}) because we have growth estimates on the reduced length functions and bounds on $l_{x}^{\tau}(x)$  by Proposition C, and thus we have a uniform lower bound on the reduced volume $\mathcal{V}_{M,g}(\tau)$ and so lower bounds on $\kappa$-noncollapseness at $(x,t)$.
\end{proof}

\section{Non-Collapsing and Gradient Behavior of Shrinking Solitons}
We prove theorems \ref{mthm3} and \ref{mthm4} in this section.  The key noncollapsing result we will need is the following theorem of Perelman's. In fact, Perelman's theorem is stronger than the following, but it will suffice for our purposes.

\begin{thm}\label{vol_noncol}
Let $(M,g(t))$, $t\in [0,T]$, be a Ricci flow of complete Riemannian Manifolds and let $(x,T)\in M\times\{T\}$. Let $0<\tau\leq T$ and $V>0$.  Then there exists $\kappa=\kappa(n,V)$ such that if $0<r\leq \sqrt{\tau}$ with the property that $|Rm|\leq r^{-2}$ on $P(x,T,r)$ and the reduced volume $\mathcal{V}_{M}(\tau)\geq V$, then $Vol_{g(T)}(B_{r}(x))\geq \kappa r^{n}$.
\end{thm}

Now we state our first main result of this section:

\begin{thm}
Let $(M,g,f)$ be a normalized complete shrinking soliton with bounded
curvature.  Then there exists $\kappa=\kappa(n,Vol_{f}(M))$ such that the associated
Ricci flow $(M,g(t))$ is $\kappa$-noncollapsed.
\end{thm}
\begin{proof}

Let $(x,T)\in M\times(-\infty,0)$.  Since $M\times\{T\}$ differs from $M\times\{-1\}$ by rescaling and a diffeomorphism we may assume $T=-1$.  Pick $r>0$ such that  $|Rm|\leq r^{-2}$ on $P(x,-1,r)$.  We will show $\exists$ $V = V(Vol_{f}(M))$ such that the reduced volume $\mathcal{V}_{M}(\tau)\geq V > 0$ for all $\tau$. Then we can apply the last theorem to finish ours.

Consider the sequence $(M,\tau^{-1}g(\tau t),(x,-1))$ as $\tau\rightarrow \infty$.  We will show $(M,\tau^{-1}g(\tau t),x)\rightarrow (M,g(t),p)$ for some $p\in M$.  Given this for the moment, we let $\bar{l}_{i}$ be the singular reduced length functions to $(x,0)$ in $(M,g_{i}(t),x)$ where $g_{i}(t)=\tau_{i}^{-1}g(\tau_{i}t)$ with $\tau_{i}\rightarrow\infty$ .  As in the proof of theorem \ref{asym_sol} we see $\bar{l}_{i}\rightarrow \bar{l}$, where $\bar{l}$ is a normalized shrinking soliton function.  By the estimates of proposition (C) and the bounded curvature we again get that  $\mathcal{V}_{M}(\tau)\rightarrow \int_{M}e^{-\bar{l}}dv_{g}$, a constant in $\tau$.  But by lemma \ref{f_vol} we showed the reduced volume of any two normalized solitons on the same isometry class are the same, hence $\int_{M}e^{-\bar{l}}dv_{g} = \int_{M}e^{-f}dv_{g}$ and we are done.

To prove convergence of $(M,\tau^{-1}g(-\tau),x)$ we begin by noting that $(M,\tau^{-1}g(-\tau))$ differs from $(M,g(-1))$ by a diffeomorphism.  That is, for the 1-parameter family of diffeomorphisms $\phi_{t}$ generated by $\frac{1}{|t|}\nabla_{t=-1}f_{-1}$ with $\phi_{-1}=id$ we have that $g(-1)=(\phi_{-\tau}^{-1})^{*}(\tau^{-1}g(-\tau))$.  Hence the sequence $(M,\tau^{-1}g(-\tau),x)$ can be identified with the sequence $(M,g(-1),\phi_{-\tau}(x))$ after we change by a diffeomorphism.  Now $x$ lies on a unique integral curve of $f$, and flowing backwards let $p$ be the first critical point encountered along the flow (as usual, such a $p$ exists because $f$ is proper).  Then if we can show $ d_{g(-1)}(\phi_{-\tau}(x),p)$ remains bounded then the sequence $(M,g(-1),\phi_{-\tau}(x))$ must converge isometrically to $(M,g(-1))$, as claimed.  But by the choice of $p$, $\phi_{-\tau}(x)$ is in fact converging to $p$.  Of course the same argument works on any time slice and hence we are done.

\end{proof}

\begin{remark}
As a consequence of the above we have that the associated Ricci flow to a normalized shrinking soliton with bounded curvature is automatically $(C,\kappa)$-controlled.
\end{remark}

Next we show a shrinking soliton is gradient.

\begin{thm}
Let $(M,g,X)$ a shrinking soliton with bounded curvature.  Then there exists a smooth $f$ such that $(M,g,f)$ is a gradient shrinking soliton.
\end{thm}
\begin{proof}
First assume $M$ is noncompact and let $p\in M$.  By remark (\ref{nongrad_est}) $<X,\nabla d>\rightarrow\infty$ uniformly and hence for large $r$ $<X,\nabla d>|_{\partial B_{r}(p)} > 1$.  Let $x\in B_{r}(p)$. Let $(M,g(t))$ associated Ricci flow, $\tau_{i}\rightarrow\infty$ and $g_{i}(t)=\tau_{i}^{-1}g(\tau_{i}t)$.  Then we see as in the last theorem that since $<X,\nabla d> > 1$ on $\partial B_{r}(p)$ that if $x\in B_{r}(p)$ then $\phi_{-\tau}(x)\in B_{r}(p)$ for $\tau\rightarrow\infty$.  Hence we see $(M,g_{i}(t),x)\rightarrow (M,g(t),x)$ since $x$ remains unchanged by the generating diffeomorphisms.  If $\bar{l}_{i}^{\tau}$ are the singular reduced length functions we see by the same arguments $\bar{l}_{i}^{\tau}\rightarrow\bar{l}^{\tau}$, a normalized shrinking soliton structure on $(M,g)$.

If $M$ is compact then the above is even easier since $d(x,\phi_{-\tau}(x))\leq diam(M)$ $\forall \tau$, and hence $(M,g(t),\phi_{-\tau}(x))\rightarrow (M,g(t))$.  The rest is then the same as above.
\end{proof}

\section{Asymptotic Solitons}

We will apply the previous estimates on the singular reduced length function's to understand the geometry of infinity of a shrinking soliton.  To begin with we will be able to use the soliton function directly to understand part of this behavior:

\begin{lem}
Let $(M,g)$ be a smooth complete Riemannian Manifold.  Let $f:M\rightarrow\mathds{R}$ be a smooth function with $|\nabla^{2}f|\leq C$ and $|\nabla f|(x)\rightarrow\infty$ as $x\rightarrow\infty$.  Then $\forall$ sequence $x_{n}\rightarrow\infty$ such that $(M,g,x_{n})\stackrel{C^{1,\alpha}-CG}{\rightarrow}(M_{\infty},g_{\infty},x_{\infty})$ we have that $(M_{\infty,g_{\infty}})\approx(\mathds{R},ds^{2})\times(N,h)$ splits isometrically.
\end{lem}

\begin{proof}
Let $x_{n}$ be such a sequence.  Let $f_{n}(x)=\frac{f(x)-f(x_{n})}{|\nabla f|(x_{n})}$ for $n$ large enough that $|\nabla f|(x_{n})>0$.  Then $f_{n}(x_{n})=0$, $|\nabla f_{n}|(x_{n})=1$ and $|\nabla^{2}f_{n}|\leq \frac{C}{|\nabla f|(x_{n})}$.  Hence $\exists$ $f_{\infty}\in C^{1,1}(M_{\infty})$ such that $f_{n}\stackrel{C^{1,\alpha}- CG}{\rightarrow}f_{\infty}$ with $f_{\infty}(x_{\infty})=0$, $|\nabla f_{\infty}|(x_{\infty})=1$ and $|\nabla^{2}f_{\infty}|=0$ in lipschitz sense on $\nabla f_{\infty}$ because such bounds pass to the limit, hence $f_{\infty}$ is in fact smooth and $|\nabla^{2}f_{\infty}|=0$ in smooth sense.  Further since $f_{\infty}\neq const$ ($|\nabla f_{\infty}|(x_{\infty})=1$) we have that $f_{\infty}$ must be linear and hence $(M_{\infty},g_{\infty})$ splits isometrically.
\end{proof}

\begin{prop}
Let $(M,g,f)$ be a normalized shrinking soliton with $|Rm|\leq C$, and let $(M,g(t))$ be its corresponding $(C,\kappa)$-controlled Ricci flow.  Then $\forall x_{k}\rightarrow\infty$ $\exists$ a subsequence $\{x_{n}\}$,  a complete Riemannian manifold $(M_{\infty},g_{\infty},x_{\infty})$ and a $(C,\kappa)$-controlled Ricci flow $(M_{\infty},g_{\infty}(t),x_{\infty}))$ such that $(M,g,x_{n})\stackrel{C^{\infty}- CG}{\rightarrow}(M_{\infty},g_{\infty},x_{\infty}))$, $(M,g(t),x_{n})\stackrel{C^{\infty}- CG}{\rightarrow}(M_{\infty},g_{\infty}(t),x_{\infty}))$, $(M_{\infty},g_{\infty})=(M_{\infty},g_{\infty}(-1))$ and $(M_{\infty},g_{\infty}(t))$ splits $(\mathds{R},ds^{2})\times(N,h(t))$ where $(N,h(t))$ is $(C,\kappa')$-controlled Ricci flow.
\end{prop}

\begin{proof}
Let $x_{k}\rightarrow\infty$.  Since $|Rm[g(t)]|\leq \frac{C}{|t|}$, the existence of limit $(M_{\infty},g_{\infty},x_{\infty})$ and $(M_{\infty},g_{\infty}(t),x_{\infty})$ is immediate from compactness and Theorem \ref{mthm3}.  Because of lemma \ref{growth_sol} the soliton functions $f_{t}=\phi^{*}_{t}f$ for $(M,g(t))$ satisfy the conditions of the last lemma, so $(M_{\infty},g_{\infty}(t),x_{\infty})$ splits $\forall$ $t$, and by the existence result of ([Sh]) and the uniqueness result of ([CZ]) so must the Ricci flow.  If $(M_{\infty},g_{\infty}(t))$ is $\kappa$-noncollapsed, then $(N,h(t))$ is $\kappa'$-noncollapsed for $\kappa'$=$\kappa'(\kappa,n)$.
\end{proof}

\begin{prop}
Let $(M,g,f)$ be a normalized shrinking soliton with $|Rm|\leq C$.  Then $\forall$ $\{x_{k}\}\rightarrow\infty$ $\exists$ subsequence $\{x_{n}\}$ and sequences $\{x^{+}_{n}\}, \{x^{-}_{n}\}$ such that $(M,g,x_{n})\rightarrow(\mathds{R},ds^{2})\times(N,h,y)$, $(M,g,x^{-}_{n})\rightarrow(\mathds{R},ds^{2})\times(N^{-},h^{-},y^{-})$, $(M,g,x^{+}_{n})\rightarrow(\mathds{R},ds^{2})\times(N^{+},h^{+},y^{+})$ where $(N,h,y)$ is a $(C,\kappa')$-controlled Ricci flow, and $(N^{-},h^{-},y^{-})$, $(N^{+},h^{+},y^{+})$ are shrinking solitons.  If $(N^{-},h^{-},y^{-})$ and $(N^{+},h^{+},y^{+})$ are isometric, then $(N,h,y)$ is
also isometric to $(N^{\pm},h^{\pm},y^{\pm})$, and hence a shrinking soliton.  If $x_{k}$ are chosen along a unique integral curve $\gamma$ of $f$ then $\{x^{+}_{n}\}, \{x^{-}_{n}\}$ can be chosen along this integral curve as well.
\end{prop}

\begin{proof}
Because the associated Ricci flow of a $\kappa$-noncollapsed shrinking soliton, Theorem \ref{mthm3},  with bounded curvature is $(C,\kappa)$-controlled, we can pass to a subsequence $\{x_{n}\}$ such that $(M,g(t),x_{n})$ converges to a $(C,\kappa)$-controlled Ricci flow $(\mathds{R},ds^{2})\times(N,h(t),y)$.  We can pick $\tau_{k}\rightarrow\infty$ or $\tau_{k}\rightarrow 0$ such that  $(N,\tau^{-1}_{k}h(-\tau_{k}t),y)\rightarrow (N_{\infty},h_{\infty}(t),y_{\infty})$, a shrinking soliton.

Fix $\epsilon=\frac{1}{i}$.  Let $k=k(\epsilon)$ such that $(N,\tau^{-1}_{k}h(-\tau_{k}),y)$ is within $\frac{\epsilon}{2}$ of $(N_{\infty},h_{\infty}(-1),y_{\infty})$ on $B_{1/\epsilon}(y)$ in $C^{1/\epsilon}$.  Let $n=n(k)$ such that $(M,\tau^{-1}_{k}g(-\tau_{k}),x_{n})$ is within $\frac{\epsilon}{2}$ of $(\mathds{R},ds^{2})\times(N,\tau^{-1}_{k}h(-\tau_{k}),y)$ in $B_{\tau^{-1}_{k}g(-\tau_{k})}(x_{n},1/\epsilon)$ in $C^{1/\epsilon}$.  Now let $z_{i}=\phi_{-\tau_{k}}(x_{n})$, where $\phi_{t}$ are the generating diffeomorphisms for the Ricci flow.  Notice that $z_{i}$ and $x_{n}$ lie along the same integral curve of $f$ and that $B_{g(-1)}(z_{i},1/\epsilon)$ is isometric to $B_{\tau^{-1}_{k}g(-\tau_{k})}(x_{n},1/\epsilon)$.  Hence $(M,g,z_{i})$ is within $1/i$ of $(\mathds{R},ds^{2})\times(N_{\infty},h_{\infty},y_{\infty})$ on $B_{g}(z_{i},i)$ in $C^{i}$.  By letting $\tau_{k}$ tend toward zero or infinity we get our two soliton sequences, and it follows from theorem \ref{asym_sol} that if they are equal that $(N,h,y)$ is also a shrinking soliton.
\end{proof}

\begin{cor}\label{asym_int}
Let $(M,g,f)$ be a normalized shrinking soliton with $|Rm|\leq C$.  Let $\gamma$ be any integral curve of $f$ which tends to infinity.  Then there exists $\{x_{n}\}\in Image(\gamma)$ such that $(M,g,x_{n})\rightarrow (\mathds{R},ds^{2})\times(N,h,y)$ where $(N,h,y)$ is a normalized $\kappa'$-noncollapsed shrinking soliton with $|Rm|\leq C$.
\end{cor}

\begin{remark}
With a little more work one can show the shrinking soliton above is uniquely defined by the integral curve $\gamma$.  However as remark (\ref{counter}) shows the limit soliton can vary from integral curve to integral curve.
\end{remark}

\section{Geometry of Level Sets}

The following lemma can be found in the unpublished work by the author \cite{N}.  Similar estimates are also in \cite{MT}.

\begin{lem}\label{growth_sol}
Let $(M,g,f)$ be a complete soliton with $|Rc|\leq C$ and let $p\in M$.  Then there exists $a=a(n,C,|\nabla f|(p))$ and $b=b(n,C,|f|(p))$ such that $\forall x\in M$ $|\nabla f|(x)\geq <\nabla f, \nabla d>\geq \lambda d(x,p)+a$ and $f(x)\geq \frac{\lambda}{2}d(x,p)^{2}+ad(x,p)+b$.
\end{lem}

\begin{proof}
Let $\gamma:[0,d]\rightarrow M$ be a minimizing unit speed geodesic between $p$ and $x$.

First assume $d\geq 2$.  Let $E^{i}(p)$ be an orthonormal basis at $p$ with $E^{n} = \dot{\gamma}$.  Define $E^{i}(t)$ as the parallel transport of $E^{i}$ over $\gamma(t)$.  Let $h:[0,d]\rightarrow \mathds{R}$ be Lipschitz with $h(0)=h(d)=0$ and $Y^{i}(t)=h(t)E^{i}(t)$.  Since $\gamma$ is a minimizing geodesic we have by the second variation formula that

\[
0\leq \int_{0}^{d}|\nabla_{\dot{\gamma}}Y^{i}|^{2} - <R(Y^{i},\dot{\gamma})\dot{\gamma},Y^{i}>dt
\]

\[
=\int_{0}^{d}(h')^{2}|E^{i}|^{2}+h^{2}
|\nabla_{\dot{\gamma}}E^{i}|^{2}-h^{2}<R(E^{i},\dot{\gamma})\dot{\gamma},E^{i}>dt
\]

\[
=\int_{0}^{d}(h')^{2}-h^{2}<R(E^{i},\dot{\gamma})\dot{\gamma},E^{i}>dt
.
\]

Summing yields

\[
\int_{0}^{d}Rc(\dot{\gamma},\dot{\gamma}) \leq (n-1)\int_{0}^{d}(h')^{2} + \int_{0}^{d}(1-h^{2})Rc(\dot{\gamma},\dot{\gamma})
\]

Using the soliton equation and plugging in $$h(t) = \left\{\begin{array}{lr}
t &   0\leq t  \leq 1\\
1 &   1\leq t  \leq d-1\\
d-t & d-1\leq t\leq d
\end{array}\right.
 $$ we get

\[
\lambda d(x,p)-(\nabla_{\dot{\gamma}}f(\gamma(d))-\nabla_{\dot{\gamma}}f(p)) \leq 2(n-1)+2C
\]

If $d\leq 2$ then $\int_{0}^{d}Rc(\dot{\gamma},\dot{\gamma})\leq 2C$, plugging in the soliton equation yields similar estimate.  Integration over $\gamma$ yields estimate for $f$.
\end{proof}

\begin{remark}\label{nongrad_est}
The verbatim argument works if we assume $Rc+\frac{1}{2}\mathcal{L}_{X}g \geq \lambda g$ for a complete vector field $X$ to give us $<X,\nabla d> \geq \lambda d(x,p) + a$.
\end{remark}

\begin{cor}
Let $(M,g,f)$ be a complete noncompact shrinking soliton with $|Rc|\leq C$.  Then $f$ grows quadratically, is bounded below and is proper.  Outside some compact subset $f$ has no critical points.
\end{cor}

\begin{remark}\label{ffg}
In fact all that was necessary for the above growth estimates on $f$ is in the inequality $Rc+\nabla^{2}f \geq \lambda g$.  If $(M,g,f)$ satisfies such a condition with $|Rc|\leq C$ and $\lambda>0$ it follows immediately that $\int_{M}e^{-f}dv_{g}$ is finite and by lifting to the universal cover we see $(M,g)$ must have finite fundamental group.  For more on the fundamental group of such spaces also see ([N],[W]).
\end{remark}

\begin{definition}
Let $(M,g,f)$ be a normalized shrinking soliton with bounded Ricci curvature.  As a consequence of the last corollary, outside a large compact set all the level sets of the soliton function $f$ are compact, smooth diffeomorphic manifolds. Let $r_{f}$ the inf of all such $r$ for which this holds.  For each $s>r_{f}$ let $N_{s}=f^{-1}(s)$ be the soliton hypersurface.
\end{definition}

\begin{prop}
Let $(M,g,f)$ normalized shrinking soliton with $|Rm|\leq C$. Assume either $n=3$ and $Rc\geq 0$ or $n=4$ and $sec\geq 0$.  Then for $s>r_{f}$ we have that the mean curvature of $N_{s}$ is nonnegative, and positive if $Rc>0$.
\end{prop}

\begin{proof}
If $(M,g)$ is flat we are done, so we can assume otherwise.

Let $N=\frac{\nabla f}{|\nabla f|}$ be the unit normal at each point on $N_{s}$.  Recall that $\nabla_{N}R = 2Rc(N,\nabla f) \geq 0$ ($>0$ if $Rc>0$).  Let $y\in N_{s}$.  Then $\exists$ a unique integral curve $\gamma$ through $y$ which tends to infinity.  By Corollary \ref{asym_int} let $x_{n}\in\gamma\rightarrow\infty$ such that $(M,g,x_{n})\rightarrow (\mathds{R},ds^{2})\times(N,h,z)$ with $(N,h,z)$ a normalized $\kappa'$-noncollapsed shrinking soliton with $|Rm|\leq C$.

If $n=3$ and $Rc\geq 0$ then $(N,h)$ must be a quotient of the two-sphere and hence $lim R(x_{n}) = 2\lambda$ and so $R(y)\leq 2\lambda$ (strict if $Rc>0$).  If $n=4$ and $sec\geq 0$ then it follows $N$ is either a finite quotient of $S^{3}$ or $S^{2}\times\mathds{R}$, by Perelman's theorem, and so similarly $R(y)\leq 3\lambda$ (strict if $Rc>0$).

Now let $E^{i}$ orthonormal basis for the tangent space of $N_{s}$ at $y$.  Then the mean curvature satisfies

\[
H = \frac{\Sigma \nabla^{2}_{i,i}f}{|\nabla f|} = \frac{\triangle f - \nabla^{2}_{N,N}f}{|\nabla f|} = \frac{n\lambda-R-\lambda+Rc(N,N)}{|\nabla f|}
\]

\[
\geq  \frac{(n-1)\lambda-R}{|\nabla f|} \geq 0
\]

in either of our cases (strict if $Rc>0$).

\end{proof}

\begin{remark}
We will replace the $sec\geq 0$ with $Rc\geq 0$ in the $n=4$ case shortly.
\end{remark}

We get as a corollary
\begin{cor}\label{cor_mthm2}
Let $(M^{3},g,f)$ be a normalized shrinking soliton with $|Rm|\leq C$, $Rc\geq 0$.  Then $(M,g)$ is isometric to a finite quotient of $\mathds{R}\times S^{2}$, $\mathds{R}^{3}$, or $S^{3}$.
\end{cor}
\begin{remark}
This finishes the proof of Theorem \ref{mthm2}.
\end{remark}

\begin{proof}
If $(M^{3},g)$ is compact then it follows from Hamilton that $(M,g)$ is a finite quotient of the three sphere.  So we may assume noncompact and nonflat.

If $Rc$ has zeros then it follows from Hamilton's maximum principle ([H]), since $dim(M)=3$, that after passing to a finite cover they split off, and hence the result.  We wish to show $Rc>0$ is not possible. Let $x_{n}\rightarrow\infty$ along an integral curve of $f$ such that $(M,g,x_{n})\rightarrow (\mathds{R},ds^{2})\times(N,h,y)$, where $(N,h,y)$ is a nonflat shrinking soliton and hence is a quotient of the 2-sphere $S^{2}$.  Let $x_{n}\in f^{-1}(s_{n})$.  Then by the last proposition we see $Vol(N_{s_{n}})<Vol(N)$.  Now let $z\in N_{s_{n}}$, and let $Y^{i}$ orthonormal diagonalizing basis for $\nabla^{2}f|_{TN_{s_{n}}}$ and $N=\frac{\nabla f}{|\nabla f|}$ be the normal unit vector.  By the Gauss equation we see

\[
sec^{N_{s_{n}}}(Y^{i},Y^{j}) = sec(Y^{i},Y^{j})+\frac{1}{|\nabla f|^{2}}(\nabla_{i,i}f\cdot\nabla_{j,j}f)
\]
\[
\Rightarrow R^{_{s_{n}}}_{ii}=R_{ii}-sec(Y^{i},N) + \frac{1}{|\nabla f|^{2}}(\nabla_{i,i}f(\triangle f - \nabla^{2}_{N,N}f-\nabla^{2}_{i,i}f))
\]
\[
\Rightarrow R_{N_{s}} = R-2R_{NN} + \frac{1}{|\nabla f|^{2}}((\triangle f - \nabla^{2}_{N,N}f)^{2}-|\nabla^{2}f|^{2}+(\nabla^{2}_{N,N}f)^{2})
\]
\[
\leq 3\lambda - \triangle f - (\lambda - \nabla^{2}_{N,N}f) - R_{N,N}+\frac{H}{|\nabla f|}(\triangle f - \nabla^{2}_{N,N}f)
\]

since $H\geq 0$ ($\Rightarrow \triangle f - \nabla^{2}_{N N}f\geq 0$) and bounded we can let $s$ large enough that $\frac{H}{|\nabla f|}\leq 1$.  Then we get

\[
\leq 2\lambda -R_{N,N} < 2\lambda
\]

and hence this contradicts the Gauss-Bonnet.
\end{proof}

\begin{cor}
We can replace $sec\geq 0$ with $Rc\geq 0$ for $n=4$ in the previous proposition.
\end{cor}
\begin{proof}
The only place $sec\geq 0$ was used was in the classifying of $3$-solitons with $\sec\geq 0$.
\end{proof}

\begin{cor}\label{connected}
Let $(M^{4},g,f)$ be a normalized shrinking soliton with $|Rm|\leq C$ and $Rc\geq 0$.  If for $s>r_{f}$ $N_{s}$ has more than one component then $(M,g)$ is isometric to a finite quotient of $S^{3}\times\mathds{R}$
\end{cor}

\begin{proof}
If this is the case then we can construct a line and split $(M^{4},g)=(\mathds{R},ds^{2})\times(N,h)$.  If $(N,h)$ is flat then so is $M$ and hence $N_{s}$ has only one component, so may assume nonflat.  By Corollary \ref{cor_mthm2} we see $(N,h)$ is either finite quotient of $S^{3}$ or $S^{2}\times\mathds{R}$.  If the latter then the level sets for $f$ are connected.
\end{proof}

\section{Asymptotic Solitons in dimension four}

Finally we are in a position in dimension four to understand the asymptotic behavior of certain shrinking solitons.

\begin{lem}
Let $(M^{4},g,f)$ be a normalized shrinking soliton with $|Rm|\leq C$ and $Rc\geq 0$.  Assume $\exists$ $x_{n}\rightarrow\infty$ such that $(M,g,x_{n})\rightarrow\mathds{R}\times S^{3}/\Gamma$, then $\forall$ sequence $\{y_{n}\}\rightarrow\infty$ $\exists$ a subsequence such that $(M,g,y_{n})\rightarrow \mathds{R}\times S^{3}/\Gamma$
\end{lem}

\begin{proof}
We can assume $N_{s}$ is connected for $s>r_{f}$, since otherwise we are done by Corollary \ref{connected}.

First let us assume $y_{n}$ is any sequence such that $(M,g,y_{n})\rightarrow \mathds{R}\times(N,h)$ where $(N,h)$ is a shrinking soliton.  Let $y_{n}\in N_{s_{n}}$ and $x_{n}\in N_{r_{n}}$.  Note that $Vol(N_{s})$ is nondecreasing, and hence must converge to $Vol(S^{3}/\Gamma)$ since this is true for $N_{r_{n}}$.  If $N=\mathds{R}\times S^{2}/\tilde{\Gamma}$ then $Vol(N_{s_{n}})\rightarrow\infty$, which is not possible.  If $N=S^{3}/\tilde{\Gamma}$ then since outside a compact set $K$ we have that $M-K$ is diffeomorphic to $\mathds{R}\times N_{s_{n}}$ and $\mathds{R}\times N_{r_{n}}$ we have that $N_{s_{n}}$ and $N_{r_{n}}$ are diffeomorphic, hence $\tilde{\Gamma}$ is conjugate to $\Gamma$.

Now let $y_{n}$ arbitrary.  After passing to subsequence we can assume it converges, and by the previous there are sequences $y^{-}_{n}$, $y^{+}_{n}$ such that $(M,g,y^{\pm}_{n})$ converge to solitons.  By the last paragraph these solitons are the same and hence $(M,g,y_{n})$ converges to the same soliton by Proposition (\ref{asym_sol}).
\end{proof}

\begin{prop}\label{asym_four}
Let $(M^{4},g,f)$ be a normalized shrinking soliton with $|Rm|\leq C$ and $Rc\geq 0$.  Fix $p\in M$.  Then if $\exists$ $x_{n}\rightarrow\infty$ such that $(M,g,x_{n})\rightarrow\mathds{R}\times S^{3}/\Gamma$, then $\forall \epsilon>0$ $\exists$ $r>0$ such that if $x\not\in B_{r}(p)$ then $B_{1/\epsilon}(x)$ must be within $\epsilon$ of $\mathds{R}\times S^{3}/\Gamma$ in $C^{1/\epsilon}$
\end{prop}

\begin{proof}
Assume for some $\epsilon$ that no such $r$ exists.  Let $y_{n}\rightarrow\infty$ such that $B_{1/\epsilon}(y_{n})$ are not $\epsilon$ close to $\mathds{R}\times S^{3}/\Gamma$.  But then by the last lemma a subsequence converges to $\mathds{R}\times S^{3}/\Gamma$, a contradiction.
\end{proof}

\begin{lem}\label{asym_four2}
Let $(M^{4},g,f)$ be a normalized shrinking soliton with $|Rm|\leq C$ and $Rc\geq 0$.  Assume $\exists$ $x_{n}\rightarrow\infty$ such that $(M,g,x_{n})\rightarrow\mathds{R}^{2}\times S^{2}/\Gamma$, then $\forall$ integral curve $\gamma\rightarrow\infty$ of $f$ $\exists$ $y_{n}\in\gamma\rightarrow\infty$ such that $(M,g,y_{n})\rightarrow \mathds{R}^{2}\times S^{2}/\tilde{\Gamma}$
\end{lem}

\begin{proof}
We know by Corollary \ref{asym_int} that $\exists$ $y_{n}\in\gamma$ tending to infinity such that $(M,g,y_{n})\rightarrow \mathds{R}\times N$, for $N$ a shrinking soliton.  If $N$ is not a quotient of $\mathds{R}^{2}\times S^{2}$ then by Corollary \ref{cor_mthm2} it is a quotient of $S^{3}$ and by the last lemma then contradicts the existence of $x_{n}$.
\end{proof}

\begin{remark}\label{counter}
Notice that in the above $\Gamma$ and $\tilde{\Gamma}$ may be different.  Let $N=\mathds{R}\times(\mathds{R}\times S^{2})/\Gamma$ where $\Gamma=\mathds{Z}_{2}$ acts by $(x,s)\rightarrow (-x,-s)$, the antipodal map.  Then if $x_{n}$ tend to infinity along the first $\mathds{R}$ factor the limit is still $\mathds{R}^{2}\times S^{2}/\Gamma$, while if $x_{n}$ tend to infinity along the second $\mathds{R}$ factor the limit is $\mathds{R}^{2}\times S^{2}$.
\end{remark}

\section{Shrinking solitons with positive soliton hessian}

The results of this section are similar to those proved by the author in the unpublished work by the author \cite{Na}.  For other interesting results in this direction see \cite{PW}.

\begin{thm}\label{hess_sol}
Let $(M,g,f)$ be complete shrinking soliton.  Assume $Rc\geq 0$ and
that $\nabla^{2}f>0$, then $(M,g)$ is isometric to $(\mathds{R}^{n},g_{0})$
\end{thm}

To prove the above we introduce the notion of the $f$-Laplacian of a
function $u$.  The motivation is fairly clear and comes directly from the standard
Laplace-Beltrami operator, which is defined as $\triangle =
\nabla^{*}\nabla$ with $\nabla^{*}$ the adjoint of the covariant
derivative with respect to the Riemannian volume form. Similarly we
define:

\begin{definition}
The $f$-Laplacian of a function $u$ is defined by
$\triangle_{f}u\equiv \nabla^{*f}\nabla u$, where the adjoint is
taken with respect to the $f$-measure $e^{-f}dv_{g}$.
\end{definition}

We will use the following simple lemma

\begin{lem}
Let $(N,g,f)$ be a complete shrinking soliton which is Ricci flat, then $(N,g)$ is isometric to $(\mathds{R}^{n},g_{0})$.
\end{lem}

\begin{proof}

We know by Ricci flatness that on $N$, $\nabla^{2}f=\lambda g$.  Now
by our growth estimate we know $f$ always has a global minimum point, say $p\in
N$.  If $x\in N$ and $\gamma$ a geodesic connecting $x$ to $p$ we
see by integration that $\nabla_{\dot{\gamma}}f(x)=\lambda d(x,p)$
and $f(x)= \frac{\lambda}{2}d(x,p)^{2}+f(p)$ . Hence $f$ has a
unique nondegenerate minimum point at $p$.  Now we compute

\[
R_{ijkq}\nabla^{q}f = (\nabla_{i}\nabla_{j}\nabla_{k}f -
\nabla_{j}\nabla_{i}\nabla_{k}f) = \nabla_{j}R_{ik}-\nabla_{i}R_{jk}
= 0. \label{Rm}
\]

In particular, because $p$ is a nondegenerate critical point, for any
unit vector $X\in T_{p}M$ we can find $x\rightarrow p$ such that
$\frac{\nabla f(x)}{|\nabla f|}\rightarrow X$ (just use Taylor's
theorem in normal coordinates to see this). Dividing both sides of
the above by $|\nabla f|$, taking $X=\partial_{l}$ and limiting out
we get that $Rm(p)=0$. A final computation now gives us that
\[
\nabla_{\nabla f}|Rm|^{2} = \nabla^{p}f\nabla_{p}|Rm|^{2} =
2\nabla^{p}f R^{ijkl}\nabla_{p}R_{ijkl}
\]
\[
= -2\nabla^{p}f R^{ijkl} (\nabla_{i}R_{jpkl} + \nabla_{j}R_{pikl})
\]
\[
= -2R^{ijkl}(\nabla_{i}( \nabla^{p}f R_{jpkl})+\nabla_{j}(
\nabla^{p}f
R_{pikl})-R_{jpkl}\nabla_{i}\nabla^{p}f-R_{pikl}\nabla_{j}\nabla^{p}f)
\]
\[
= 4\lambda R^{ijkl}(R_{jikl}) = -4\lambda |Rm|^{2} \leq 0
\]

Our explicit formula for $f$ tells us that the negative gradient
flow from any $x\in N$ converges to $p$, and hence from the above
$|Rm|$ takes a maximum at $p$.  But we showed $Rm(p)=0$.  Hence
$Rm=0$.  Because $f$ is smooth and has a unique critical minimum we see $N$ is homeomorphic to $\mathds{R}^{n}$.  Since $N$ is simply connected and flat, $N$ is isometric to $\mathds{R}^{n}$.
\end{proof}

Under a positivity assumption on $Rc+\nabla^{2}f$ we have the following
estimate and Liouville type theorem:

\begin{prop}
Let $(M,g,f)$ be smooth, complete with bounded Ricci curvature. Assume
$Rc+\nabla^{2}f \geq \lambda g$ with $\lambda > 0$.  Let
$u:M\rightarrow\mathds{R}$ smooth.  Then

1) There exists $\alpha>0$ such that if $\triangle_{f}u=0$ and $|u|\leq
A e^{\alpha d(x,p)^{2}}$ for  some $A>0$ and $p\in M$, then $u=constant$.

2) If $\triangle_{f}u\geq 0$ and $u$ is bounded then $u=constant$.
\end{prop}
\begin{remark}
The above need not hold for $\lambda\leq 0$.
\end{remark}
\begin{proof}[Proof of (1)]
Let $x\in M$ be arbitrary.  Note by multiplying by $e^{-f}$ we get
\begin{equation}
\nabla^{i}(e^{-f}\nabla_{i}u)=0.
\end{equation}
Let $\phi: M\rightarrow\mathds{R}$ be a cutoff function with
$$\phi = \left\{\begin{array}{lr}
1 \textrm{ on }  B(x,1)\\
0\leq\phi\leq 1 \textrm{ on } B(x,1+r)-B(x,1)\\
0 \textrm{ on } M-B(x,1+r)
\end{array}\right.
 $$
where $r>0$ and $|\nabla\phi|\leq\frac{C}{r}$ for some $C$.
Multiplying the above by $\phi^{2}u$ and integrating we get
\[
-\int(2\phi u\nabla^{i}\phi\nabla_{i}u+\phi^{2}|\nabla u|^{2})
e^{-f}dv_{g} = 0
\]
\[
\int\phi^{2}|\nabla u|^{2} e^{-f}dv_{g} = -2\int(\phi
u\nabla^{i}\phi\nabla_{i}u) e^{-f}dv_{g}
\]
\[
\leq \int (\frac{1}{2}\phi^{2}|\nabla u|^{2}
+2u^{2}|\nabla\phi|^{2}) e^{-f}dv_{g}
\]
so that
\[
\int_{B(x,1)}|\nabla u|^{2} e^{-f}dv_{g}\leq 4\int_{M}u^{2}|\nabla
\phi|^{2} e^{-f}dv_{g}
\]
\[
\leq \frac{4C^{2}}{r^{2}}\int_{B_{1+r}-B_{1}}u^{2}e^{-f}dv_{g} \leq
\frac{4C^{2}}{r^{2}}\int_{M}u^{2}e^{-f}dv_{g}.
\]
But let $\alpha<\frac{\lambda}{4}$, and thus $u^{2}(x)\leq
A e^{2\alpha d(x,p)^{2}}$. So in exponential coordinates we compute
\[
\int_{M}u^{2}e^{-f}dv_{g}\leq\int_{S^{n-1}}\int_{0}^{\infty}
e^{-(\frac{\lambda}{2}-2\alpha)r^{2}+a r+b}dr ds_{n-1} < \infty
\]
for some constants $a$ and $b$ by lemma \ref{growth_sol}.  Thus we can tend
$r\rightarrow\infty$ to get
\[
\int_{B(x,1)}|\nabla u|^{2} e^{-f}dv_{g} = 0
\]
Since $x$ was arbitrary, $|\nabla u|=0$ and thus $u$=constant.
\end{proof}
\begin{proof}[Proof of (2)]
This is much the same.  Since $u$ is bounded above we can assume, by
adding a constant, that $sup$ $u = 1$. Let $u^{+}(x)=max(u(x),0)$.
Let $x\in M$ such that $u(x)>0$ and $\phi$ as in the last part with
center $x$. Then our equation $\nabla^{i}(e^{-f}\nabla_{i}u)\geq 0$
gives

\[
-\int(2\phi u^{+}\nabla^{i}\phi\nabla_{i}u+\phi^{2}\nabla^{i}
u^{+}\nabla_{i} u) e^{-f}dv_{g} \geq 0
\]
so that
\[
\int_{M}\phi^{2}|\nabla u^{+}|^{2} e^{-f}dv_{g} \leq
\frac{4C^{2}}{r^{2}}\int_{M}(u^{+})^{2}e^{-f}dv_{g}
\]
But $u^{+}$ is bounded and $\int_{M}e^{-f}dv_{g}$ is finite.  So we
may limit out, using monotone convergence, to get $\int_{M}|\nabla
u^{+}|^{2} e^{-f}dv_{g}=0$.  So $u^{+}$ is constant.  Since
$u(x)>0$, $u$ is constant.
\end{proof}

As a consequence of the above we get the following

\begin{lem}
Let $(M,g,f)$ complete shrinking soliton with $Rc,\nabla^{2}f\geq 0$, then the eigenvalues of $Rc$ are all either $0$ or $\lambda$.
\end{lem}

\begin{proof}
We have $Rc\leq Rc+\nabla^{2}f = \lambda g$.

The following computation is useful:
\[
\nabla^{i}R_{ij}+\nabla^{i}\nabla_{i}\nabla_{j}f=0
\]
\[
\frac{1}{2}\nabla_{j}R + \nabla_{j}(-R-n\lambda) +
Rc_{jk}\nabla^{k}f = 0
\]
\begin{equation}
\nabla_{i}R = 2Rc_{ij}\nabla^{j} f
\end{equation}
Now if we take the divergence of this we get
\begin{equation}
\triangle_{f}R = 2(\lambda R-|Rc|^{2})
\label{R_eq}
\end{equation}

Now if $\partial_{i}$ is an eigenbasis for $Rc$ we write the rhs of
(\ref{R_eq}) as $(\lambda R-|Rc|^{2}) = \Sigma
R_{ii}(\lambda-R_{ii})\geq 0$ under our assumptions.  In particular
the scalar curvature is a bounded subsolution to $\triangle_{f}$,
and thus must be constant.  Plugging this in we see that $\Sigma
R_{ii}(\lambda-R_{ii})= 0$, which under our assumptions implies that
each term is zero and thus every eigenvalue of $Rc$ is either $0$ or
$\lambda$.
\end{proof}

Now we prove Theorem \ref{hess_sol}:

\begin{proof}[Theorem \ref{hess_sol}]
If $\nabla^{2}f>0$ then we must have that $0\leq Rc < \lambda$.  By the last lemma we must then have that $Rc=0$.  We proved a Ricci Flat shrinking soliton is isometric to $(\mathds{R}^{n}),g_{0}$, and hence we are done.
\end{proof}

\section{Shrinking Solitons where the Ricci Tensor has a zero}

The goal of this section is to prove the following.

\begin{thm}
Let $(M^{4},g,f)$ be a shrinking soliton with bounded
nonnegative sectional curvature.  Assume the Ricci Tensor has a zero
eigenvalue at some point, then $(M,g)$ is isometric to a finite
quotient of $S^{2}\times\mathds{R}^{2}$, $S^{3}\times\mathds{R}$ or $\mathds{R}^{4}$.
\end{thm}

\begin{proof}
Let $p\in M$, $v\in T_{p}M$ such that $v$ is a zero eigenvalue of $Rc$.  If $e^{i}$ is an orthonormal eigenbasis for $Rc$ at $p$ we see that $\frac{\partial Rc}{\partial t}(v,v) = \triangle Rc(v,v) +2(R_{ivjv}R^{ij}-R_{vi}R^{i}_{ v})$, where then $R_{ivjv}R^{ij}-R_{vi}R^{i}_{ v} = R_{iviv}R^{ii}-0=\Sigma sec(v,i)R_{i i} \geq 0$.  Hence since $g(t)$ only change by scaling and diffeomorphisms, if $\exists$ at zero of the Ricci at some point and time then there is a zero at every point in time, and we can apply Hamilton's maximum principle as in \cite{Ha} to see there must be a zero at every point with the zero eigenspace parallel translation invariant.  After passing to the universal cover (which is finite by remark\ref{ffg}) we see the soliton splits and we can use Perelman's classification in dimension three.
\end{proof}

\section{Four Solitons with Positive Curvature}

We finish the proof of the main theorem with the following.

\begin{lem}
Let $(M^{4},g,f)$ be a shrinking soliton with bounded nonnegative curvature operator and $Rc>0$.  Then $M$ is compact.
\end{lem}
\begin{proof}
Assume $M$ is noncompact.  Let $x_{n}\rightarrow\infty$ such that $(M,g,x_{n})\rightarrow\mathds{R}\times N$ for $N$ a shrinking soliton (such $x_{n}$ exist by Corollary \ref{asym_int}).  Let $F=f^{-1}(-\infty,r)$ where $r>r_{f}$, so that $\partial F=N_{r}$.  We can assume $N_{r}$ has only one component, else we have that $M$ has more than one end and we can split, contradicting strictly positive curvature.  We have two cases:

First assume $N$ is a finite quotient of $\mathds{R}\times S^{2}$.  Then outside $F$, when $f$ has no more critical points, we see that since $R$ is increasing along the integral curves of $f$ then by lemma \ref{asym_four2} we have that $R< 2\lambda$ and hence $Rc\leq R/2< \lambda$.  In particular outside this compact set we have $\nabla^{2}f=\lambda g-Rc>0$.  We show $\nabla^{2}f>0$ on all of $M$, and then it will follow Theorem \ref{hess_sol} that $(M,g)$ is isometric to $(\mathds{R}^{n},g_{0})$, and hence can't have strictly positive curvature and is a contradiction.  Define a map $h:\partial F\rightarrow M$ as follows:  For $x\in\partial F$, $x$ lies in a unique integral curve of $f$.  We can follow the integral curve backwards, and since $F$ compact, we hit a unique critical point $p_{x}$ of $f$.  We define $h(x)=p_{x}$.  Now note that if $x\in\partial F$ then $\nabla^{2}f(h(x))\geq \lambda g-Rc\geq \frac{\lambda}{2}(2\lambda-R)g>\delta>0$, since $R$ only decreases as we move backwards along an integral curve.  So the gradient flow is attracting at $p$ and hence if $h(x)=p$, then by the continuity of initial conditions $\exists \epsilon>0$ such that if $y\in\partial F$ with $d(x,y)<\epsilon$ then $h(y)=h(x)=p$.  Hence $h^{-1}(p)\equiv D$ is open.  But let $z\in \partial F$ be in the closure of $D$.  The same argument tells us $h^{-1}(h(z))$ is also open, and since $z$ is in the closure of $D$ we have $h(z)=p$.  Hence $D$ is closed.  Since $\partial F$ is connected we have $h(x)=p$ $\forall x\in \partial F$.  Hence there is only one critical point $p$ for $f$ and $\nabla^{2}f>0$.

Now assume $N$ is a finite quotient of $S^{3}$.  Pick $\epsilon>0$ small, $p\in M$, and by proposition (\ref{asym_four}) let $s>0$ such that if $x\not\in B_{s}(p)$ then $(M,g,x)$ is $\epsilon$ close to $\mathds{R}\times N$.  In particular, $\exists \delta>0$ such that for all $x\not\in B_{s}(p)$ the isotropic curvature of $(M,g)$ at $x$ is bounded uniformly from below by $\delta$.  If $x\in B_{s}(p)$ then certainly the isotropic curvature is nonnegative since $Rm\geq 0$.  If there is a zero at some point then by Hamilton's result \cite{Ha} we see there must be a zero everywhere, which is not true.  Hence the isotropic is positive inside $B_{s}(p)$ and uniformly positive on all of $M$.  Using the result by \cite{NiWa}, we see $(M,g)$ is isometric to a finite quotient of  $\mathds{R}\times S^{3}$, hence does not have strictly positive curvature, also a contradiction.
\end{proof}

\section*{Acknowledgements} The author would like very much to thank Gang Tian
for many helpful discussions and encouragement during this project, as well as Andre Neves for helpful conversations.  Additionally the author is grateful to Hans-Joachim Hein and Richard Bamler for help with the preparing of and discussing of this paper.


Department of Mathematics, Princeton University, Princeton, NJ 08544 

email: {\tt anaber@math.princeton.edu}
\end{document}